\documentclass[12pt]{article}

\usepackage{amsmath,amsthm,amsfonts,amssymb,latexsym,amscd}

 \input{prepictex.tex}
 \input{pictex.tex}
 \input{postpictex.tex}

\begin{document}

\newtheorem{theorem}{Theorem}[section]
\newtheorem{corollary}[theorem]{Corollary}
\newtheorem{lemma}[theorem]{Lemma}
\newtheorem{proposition}[theorem]{Proposition}
\newtheorem{conjecture}[theorem]{Conjecture}
\newtheorem{definition}[theorem]{Definition}
\newtheorem{problem}[theorem]{Problem}
\newtheorem{remark}[theorem]{Remark}
\newtheorem{example}[theorem]{Example}

\def\mA{{\mathfrak A}}
\def\mB{{\mathfrak B}}

\def\A{{\mathcal A}}
\def\B{{\mathcal B}}
\def\C{{\mathcal C}}
\def\D{{\mathcal D}}
\def\E{{\mathcal E}}
\def\F{{\mathcal F}}
\def\H{{\mathcal H}}
\def\J{{\mathcal J}}
\def\K{{\mathcal K}}
\def\LL{{\mathcal L}}
\def\N{{\mathcal N}}
\def\O{{\mathcal O}}
\def\P{{\mathcal P}}
\def\R{{\mathcal R}}
\def\SS{{\mathcal S}}
\def\U{{\mathcal U}}
\def\W{{\mathcal W}}
\def\Z{{\mathcal Z}}

\def\bC{{\mathbb C}}
\def\bN{{\mathbb N}}
\def\bT{{\mathbb T}}
\def\bZ{{\mathbb Z}}

\def\auto{{\operatorname{Aut}}}
\def\endo{{\operatorname{End}}}
\def\outo{{\operatorname{Out}}}
\def\inn{{\operatorname{Inn}}}
\def\sp{{\operatorname{span}}}
\def\clsp{{\overline{\operatorname{span}}}}
\def\di{{\operatorname{diag}}}
\def\Ad{\operatorname{Ad}}
\def\id{\operatorname{id}}
\def\perm{\operatorname{Perm}}
\def\res{\operatorname{Res}}

\def\g{{\mathfrak G}{\mathfrak R}_E}
\def\gg{{\mathfrak G}_E}
\def\ae{\widetilde{{\mathfrak A}}_E}
\def\aee{{\mathfrak A}_E}

\title{Endomorphisms of Graph Algebras}

\author{Roberto Conti\footnote{This research was supported through the 
programme "Research in Pairs" by the Mathematisches Forschungsinstitut Oberwolfach in 2011.}, 
Jeong Hee Hong\footnote{This work was supported by National Research Foundation of Korea
Grant funded by the Korean Government (KRF--2008--313-C00039).}\hspace{1mm} and 
Wojciech Szyma{\'n}ski*\footnote{This work was supported by: the FNU Rammebevilling `Operator
algebras and applications' (2009--2011), the Marie Curie Research Training Network
MRTN-CT-2006-031962 EU-NCG, the NordForsk Research Network `Operator algebra and dynamics' 
(grant \#11580), the FNU Forskningsprojekt `Structure and Symmetry' (2010--2012), 
and the EPSRC Grant EP/I002316/1.}}

\date{19 May, 2012 (revised 18 August, 2012)}
\maketitle

\renewcommand{\sectionmark}[1]{}

\vspace{7mm}
\begin{abstract}
We initiate a systematic investigation of endomorphisms of graph $C^*$-algebras $C^*(E)$, extending 
several known results on endomorphisms of the Cuntz algebras $\O_n$. 
Most but not all of this study is focused on endomorphisms which permute the vertex projections and globally preserve the diagonal MASA $\D_E$ of $C^*(E)$. Our results pertain both automorphisms and 
proper endomorphisms. Firstly, the Weyl group and the restricted Weyl group of a graph 
$C^*$-algebra are introduced and investigated. In particular, criteria of outerness for automorphisms 
in the restricted Weyl group are found. We also show that the restriction to the diagonal  MASA of an 
automorphism which globally preserves both $\D_E$ and the core AF-subalgebra eventually 
commutes with the corresponding one-sided shift. Secondly, we exhibit several properties of proper 
endomorphisms, investigate invertibility of localized endomorphisms both on $C^*(E)$  
and in restriction to $\D_E$, and develop a combinatorial approach to analysis of permutative endomorphisms. 
\end{abstract}

\vfill
\noindent {\bf MSC 2010}: 46L40, 46L55, 46L05

\vspace{3mm}
\noindent {\bf Keywords}: Cuntz-Krieger algebra, graph algebra, endomorphism, automorphism, 
permutative endomorphism, AF-subalgebra, MASA, subshift

\newpage


\section{Introduction}\label{sectionI}

The main aim of this article is to carry out systematic investigations of a certain natural class of 
endomorphisms and, in particular, automorphisms of graph $C^*$-algebras $C^*(E)$. Namely, we focus on 
those endomorphisms which globally preserve the diagonal MASA $\D_E$. 
This leads to the concept of the {\em Weyl group} of a graph algebra, arising from the normalizer 
of a maximal abelian subgroup of the automorphism group of $C^*(E)$, consisting of those 
automorphisms which fix the diagonal $\D_E$ point-wise. We investigate in depth 
an important subgroup of the Weyl group corresponding to those automorphisms which 
globally preserve the core AF-subalgebra $\F_E$ as well---the {\em restricted Weyl group} 
of $C^*(E)$. We develop powerful novel techniques of both analytic and combinatorial nature for the study 
of automorphisms these groups comprise. 

\vspace{2mm}
By analogy with the theory of semi-simple Lie groups, 
Cuntz introduced in \cite{Cun2} the Weyl group of the simple, purely infinite $C^*$-algebras $\O_n$. 
Quite similarly with the classical theory, it arises as the quotient of the normalizer of a maximal abelian 
subgroup of the automorphism group of the algebra. So defined Weyl group is discrete albeit 
infinite, and the abelian subgroup in question is an inductive limit of higher dimensional tori. 
Cuntz posed a problem of determining the structure of the important subgroup of the Weyl group, 
corresponding to those automorphisms in the Weyl group which globally preserve the 
canonical UHF-subalgebra of $\O_n$, \cite{Cun2}. After 30 years, this question has been 
finally answered in \cite{CHS1}. In the present paper, we take this programme one step further, 
expanding it from Cuntz algebras $\O_n$ to a much wider class of graph $C^*$-algebras. 

The theory of graph $C^*$-algebras began in earnest in the late nineties, \cite{KPRR,KPR}, and since then 
it has developed into a fully fledged and very active area of research within operator algebras. 
For a good general introduction to graph $C^*$-algebras we refer the reader to \cite{Rae}. 
In the case of finite graphs, the corresponding $C^*$-algebras essentially coincide with 
Cuntz-Krieger algebras, that were introduced much earlier, \cite{CK}, in connection with 
topological Markov chains.
The importance of graph algebras (or Cuntz-Krieger algebras) stems to a large extent from numerous 
applications they have found. Not trying to be exhaustive in any way, we only mention: their role in
classification of purely infinite, simple $C^*$-algebras, \cite{Ror1,Szy1}, and related to that applications 
to the problem of semiprojectivity of Kirchberg algebras, \cite{Bla,Szy2,Sp}; their connection with objects of 
interest in noncommutative geometry and quantum group theory, \cite{HSz,CM,PR}; their  strong interplay 
with theory of symbolic dynamical systems, going back to the original paper of Cuntz and Krieger, \cite{CK,Boy}. 
It is also worth mentioning that graph $C^*$-algebras (similarly to Cuntz algebras) in many ways behave  
as purely combinatorial objects, and are thus strongly connected to their purely algebraic counterparts, the 
Leavitt path algebras, \cite{AAP}. Therefore we believe that many of the results of the present 
paper are applicable to the latter algebras as well. 

The analysis of endomorphisms of graph algebras, developed in the present article, owes a great 
deal to the close relationship between such algebras and the Cuntz algebras $\O_n$. The 
$C^*$-algebras $\O_n$ were first defined and investigated by Cuntz in his seminal 
paper \cite{Cun1}, and they bear his name ever since. The Cuntz algebras have been extensively 
used in many a diverse contexts, including classification of $C^*$-algebras, quantum field theory, 
self-similar sets, wavelet theory, coding theory, spectral flow, subfactors 
and index theory, among others.  

Systematic investigations of endomorphisms of $\O_n$, $2\leq n<\infty$, were initiated by Cuntz in 
\cite{Cun2}.  A fundamental bijective correspondence $\lambda_u\leftrightarrow u$ 
between unital $*$-endomorphisms and 
unitaries in $\O_n$ was established therein. Using this correspondence, Cuntz 
proved a number of interesting results, in particular with regard to those endomorphisms which 
globally preserve either the core UHF-subalgebra $\F_n$ or the diagonal MASA $\D_n$. 
Likewise, investigations of automorphisms of $\O_n$ began almost immediately after the birth 
of the algebras in question, \cite{Cun2}. 
Endomorphisms of the Cuntz algebras have played a role in certain aspects of  index  
theory, both from the $C^*$-algebraic and von Neumann algebraic point of view, e.g. see 
\cite{I1}, \cite{J} and \cite{CP}. One of the most interesting applications of 
endomorphisms of $\O_n$, found by Bratteli and J{\o}rgensen in \cite{BJ1,BJ2}, is in the 
area of wavelets. In particular, permutative endomorphisms have been used in this context. 
These were further investigated by Kawamura, \cite{Kaw}. 

The present article builds directly on the progress made recently in the study of localized endomorphisms 
of $\O_n$ by Conti, Szyma\'{n}ski and their collaborators. In particular, much better understanding 
of those endomorphisms of $\O_n$ which globally preserve the core UHF-subalgebra or the canonical 
diagonal MASA has been obtained in \cite{CRS} and \cite{HSS}, respectively. 
In \cite{Szy,CS,CKS,C,CHS2}, a novel combinatorial 
approach to the study of permutative endomorphisms of $\O_n$ has been introduced, 
and subsequently significant progress in the investigations of such endomorphisms and  
automorphisms has been obtained. In particular, a striking relationship between 
permutative automorphisms of $\O_n$ and automorphisms of the full two-sided $n$-shift has been 
found in \cite{CHS1}. 

In the present article, we extend this analysis of endomorphisms of the Cuntz 
algebras to the much larger class of graph $C^*$-algebras. Most of our results (but not all) 
are concerned with algebras corresponding to finite graphs without sinks, which may be identified 
with Cuntz-Krieger algebras of finite $0$--$1$ matrices, \cite{CK}. From the very beginning, 
the theory of such algebras has been closely related to dynamical systems. In particular, endomorphisms of 
Cuntz-Krieger algebras have been studied in the context of index theory, \cite{I2}. Quasi-free 
automorphisms of Cuntz-Krieger algebras (and even more generally, Cuntz-Pimsner algebras) have been 
studied in \cite{KT}, \cite{Z} and \cite{DS}. An interesting connection between automorphisms 
of Cuntz-Krieger algebras and Markov shifts has been investigated by Matsumoto, \cite{Mat1,Mat}. 

\medskip
The present paper is organized as follows. In Section \ref{sectionNP}, we set up notation and present some 
preliminaries. In particular, we introduce a class of endomorphisms $\{\lambda_u:u\in\U_E\}$ of a graph 
$C^*$-algebra $C^*(E)$ corresponding to unitaries commuting with the vertex projections 
(whose set we denote $\U_E$). 

In Section \ref{sectionF}, an analogue of the Weyl group for graph $C^*$-algebras is introduced 
and investigated. Namely, let $\D_E$ be the canonical abelian subalgebra of $C^*(E)$. Then, 
under a mild hypothesis, the group $\auto_{\D_E}(C^*(E))$ of those automorphisms of $C^*(E)$ 
which fix $\D_E$ point-wise is a maximal abelian subgroup of $\auto(C^*(E))$ (Propositions  
\ref{diagonalunitaries} and \ref{autonormalizer}). The Weyl group ${\mathfrak W}_E$ 
of $C^*(E)$ is defined as the quotient of the normalizer of $\auto_{\D_E}(C^*(E))$ by itself. 
Then we exhibit several structural properties of the Weyl group, analogous to those discovered by Cuntz 
in the case of $\O_n$, \cite{Cun2}. In particular, the Weyl group of $C^*(E)$ is countable and discrete 
(Proposition \ref{countableweyl}). 

In Section \ref{sectionR}, we investigate an important subgroup of the Weyl group corresponding 
to those automorphisms which globally preserve both the diagonal $\D_E$ and the core 
AF-subalgebra $\F_E$, the restricted Weyl group ${\mathfrak R}{\mathfrak W}_E$ of $C^*(E)$. 
We prove certain facts about the normalizer of $\F_E$ (Theorem \ref{normalizer}) and give a 
characterization of automorphisms globally preserving $\F_E$ (Lemma \ref{fautos}). We obtain 
a very convenient criterion of outerness for a large class of such automorphisms (Corollary \ref{outerness}). 
Then we take a closer look at the action on the diagonal $\D_E$ induced by the restriction 
of those automorphisms of the graph algebra which preserve both the diagonal and the core 
AF-subalgebra. By earlier results of Cuntz-Krieger and Matsumoto, \cite{CK,Mat1,Mat}, it is known 
that all shift commuting automorphisms of the diagonal $\D_E$ extend to automorphisms of $C^*(E)$. 
In the case of the Cuntz algebra, it was shown in \cite{CHS1} that all automorphisms of the 
diagonal which (along with their inverses) eventually commute with the shift can be extended 
to automorphisms of $\O_n$. In the present paper, we show that the automorphisms of $\D_E$ 
arising from restrictions of AF-subalgebra $\F_E$ preserving automorphisms of $C^*(E)$ 
eventually commute with the shift (Theorem \ref{eventuallycommute}). 
This provides a solid starting point for a future  complete determination
of the restricted Weyl group ${\mathfrak R}{\mathfrak W}_E$ of a graph $C^*$-algebra $C^*(E)$. 

In Section \ref{sectionL}, we study localized endomorphisms of a graph algebra $C^*(E)$, that is 
endomorphisms $\lambda_u$ corresponding to unitaries $u$ from the algebraic part of the 
core AF-subalgebra which commute with the vertex projections. 
We obtain an algorithmic criterion of invertibility of such localized endomorphisms  
(Theorem \ref{main}), as well as a criterion of invertibility of the restriction of a localized 
endomorphism to the diagonal MASA (Theorem \ref{maindiagonal}). These theorems  
extend the analogous results for Cuntz algebras obtained earlier in \cite{CS}. 
Localized endomorphisms constitute the best understood and most studied class 
of endomorphisms of the Cuntz algebras, and it is natural to begin systematic investigations 
of endomorphisms of graph $C^*$-algebras from them. 

In Section \ref{sectionP}, a special class of localized endomorphisms corresponding to permutation 
unitaries is investigated. Automorphisms of $C^*(E)$ of this type give rise to a large and 
interesting subgroup of the restricted Weyl group ${\mathfrak R}{\mathfrak W}_E$. For a 
permutative endomorphism $\lambda_u$ of 
a graph algebra $C^*(E)$, combinatorial criteria are given for $\lambda_u$ to be an automorphism of 
$C^*(E)$ or its restriction to be an automorphism of the diagonal $\D_E$ 
(Lemma \ref{conditionb}, Lemma \ref{conditiond} and Theorem \ref{combinatorialcriterion}). 
In this way, we obtain a far-reaching generalization of the techniques developed 
by Conti and Szyma{\'n}ski in \cite{CS} for dealing with permutative endomorphisms 
of the Cuntz algebras. A few examples are worked out in detail, illustrating applications of 
the combinatorial machinery developed in the present paper.  

\medskip
{\bf Acknowledgement.} We are grateful to the referee for very careful reading of the manuscript 
and a number of comments which helped to improve the presentation.  


\section{Notation and preliminaries}\label{sectionNP}

\subsection{Directed graphs and their $C^*$-algebras}

Let $E=(E^0,E^1,r,s)$ be a directed graph, where $E^0$ and $E^1$ are (countable) sets of vertices 
and edges, respectively, and $r,s:E^1\to E^0$ are range and source maps, respectively. 
The $C^*$-algebra $C^*(E)$ corresponding to a graph $E$ 
is by definition the universal $C^*$-algebra generated by mutually 
orthogonal projections $P_v$, $v\in E^0$, and partial isometries $S_e$, $e\in E^1$, 
subject to the following relations 
\begin{description}
\item{(GA1)} $S_e^*S_e=P_{r(e)}$ and $S_e^* S_f=0$ if $e\neq f\in E^1$, 
\item{(GA2)} $S_e S_e^*\leq P_{s(e)}$  for $e\in E^1$, 
\item{(GA3)} $P_v=\sum_{s(e)=v}S_e S_e^*$ if $v\in E^0$ emits finitely many 
and at least one edge. 
\end{description}
It follows from the above definition that a graph $C^*$-algebra $C^*(E)$ is unital if and only if the 
underlying graph $E$ has finitely many vertices, in which case $1=\sum_{v\in E^0}P_v$. 

A {\em path} $\mu$ of length $|\mu|=k\geq 1$ is a sequence 
$\mu=(\mu_1,\ldots,\mu_k)$ of $k$ edges $\mu_j$ such that 
$r(\mu_j)=s(\mu_{j+1})$ for $j=1,\ldots, k-1$. We also view the vertices as paths of length $0$.
The set of all paths of length $k$ is denoted $E^k$. The range and source maps naturally 
extend from edges $E^1$ to paths $E^k$. A {\em sink} is a vertex $v$ which emits 
no edges, i.e. $s^{-1}(v)=\emptyset$. A {\em source} is a vertex $w$ which receives no 
edges, i.e. $r^{-1}(v)=\emptyset$. By a {\em loop} we mean a path $\mu$ of 
length $|\mu|\geq 1$ such that $s(\mu)=r(\mu)$. We say that 
a loop $\mu=(\mu_1,\ldots,\mu_k)$ has an exit if there is a $j$ such that $s(\mu_j)$ 
emits at least two distinct edges. 

As usual, for a path $\mu=(\mu_1,\ldots,\mu_k)$ of length $k$ we denote by $S_\mu=
S_{\mu_1}\cdots S_{\mu_k}$ the corresponding partial isometry in $C^*(E)$. It is known 
that each $S_\mu$ is non-zero, with the domain projection $P_{r(\mu)}$. 
Then $C^*(E)$ is the closed span of 
$\{S_\mu S_\nu^*:\mu,\nu\in E^*\}$, where $E^*$ denotes the collection of 
all finite paths (including paths of length zero). Here we agree that $S_v=P_v$ for $v\in E^0$ 
viewed as path of length $0$. Also note that $S_\mu S_\nu^*$ is non-zero if and only if 
$r(\mu)=r(\nu)$. In that case, $S_\mu S_\nu^*$ is a partial isometry with domain and range projections 
equal to $S_\nu S_\nu^*$ and $S_\mu S_\mu^*$, respectively. 

The range projections $P_\mu=S_\mu S_\mu^*$ of all 
partial isometries $S_\mu$ mutually commute, and the abelian $C^*$-subalgebra of $C^*(E)$ 
generated by all of them is called the diagonal subalgebra and denoted $\D_E$. 
If $E$ does not contain sinks and all loops have exits then 
$\D_E$ is a MASA (maximal abelian subalgebra) in $C^*(E)$ by \cite[Theorem 5.2]{HPP}. 
We set $\D^0_E = {\rm span}\{P_v  :  v\in E^0  \}$ and, more generally, 
$\D_E^k= {\rm span}\{P_\mu  :  \mu\in E^k  \}$ for $k\geq 0$. 

There exists a strongly continuous action $\gamma$ of the circle group $U(1)$ on $C^*(E)$, 
called the {\em gauge action}, such that $\gamma_t(S_e)=tS_e$ and $\gamma_t(P_v)=P_v$ 
for all $e\in E^1$, $v\in E^0$ and $t\in U(1)\subseteq\bC$. 
The fixed-point algebra $C^*(E)^\gamma$ for the gauge action is an AF-algebra, denoted 
$\F_E$ and called the core AF-subalgebra of $C^*(E)$. $\F_E$ is the closed span of 
$\{S_\mu S_\nu^*:\mu,\nu\in E^*,\;|\mu|=|\nu|\}$. For $k\in\bN=\{0,1,2,\ldots\}$ 
we denote by $\F_E^k$ the linear span of $\{S_\mu S_\nu^*:\mu,\nu\in E^*,\;|\mu|=|\nu|= k\}$. 
For an integer $m\in\bZ$ we denote by $C^*(E)^{(m)}$ the spectral subspace of the gauge 
action corresponding to $m$. That is,  $C^*(E)^{(m)}:=\{x\in C^*(E):\gamma_z(x)=z^m x,\,\forall z\in U(1)\}$. 
In particular, $C^*(E)^{(0)}=C^*(E)^\gamma$. 


\subsection{Endomorphisms determined by unitaries}

We denote by $\U_E$ the collection of all those unitaries in the multiplier algebra $M(C^*(E))$ which 
commute with all vertex projections $P_v$, $v\in E^0$. That is 
\begin{equation}\label{ue}
\U_E:=\U((\D_E^0)'\cap M(C^*(E))). 
\end{equation}
These unitaries will play a crucial role throughout this paper. Let $u\in\U_E$. 
Then $uS_e$, $e\in E^1$, are partial isometries in $C^*(E)$ which together with 
projections $P_v$, $v\in E^0$, satisfy (GA1)--(GA3). Thus, by universality in the definition of $C^*(E)$, 
there exists a $*$-homomorphism $\lambda_u:C^*(E)\to C^*(E)$ such that\footnote{The reader should be 
aware that in some papers (e.g. in \cite{Cun2}, \cite{Szy} and \cite{CS}) a different convention: 
$\lambda_u(S_e)=u^* S_e$ is used.}
\begin{equation}\label{lambda}
\lambda_u(S_e)=u S_e \;\;\; {\rm and}\;\;\;  \lambda_u(P_v)=P_v, \;\;\;  {\rm for}\; e\in E^1,\; v\in E^0. 
\end{equation}
Clearly, $\lambda_u(1)=1$ whenever $C^*(E)$ is unital. In general, $\lambda_u$ may be neither injective 
nor surjective. However, the following proposition is an immediate consequence of the gauge-invariant 
uniqueness theorem \cite[Theorem 2.1]{BHRS}. 

\begin{proposition}\label{coreinjectivity}
If $u\in\U_E$ belongs to the minimal unitization of the core AF-subalgebra 
$\F_E$, then endomorphism $\lambda_u$ is automatically injective. 
\end{proposition}

Note that $\{\lambda_u:u\in\U_E\}$  is a semigroup with the following multiplication law:
\begin{equation}\label{convolution}
\lambda_u\circ\lambda_w = \lambda_{u*w}, \;\;\; u*w=\lambda_u(w)u. 
\end{equation}
We say $\lambda_u$ is {\em invertible} if $\lambda_u$ is an automorphism of $C^*(E)$. 
For $K\subseteq\U_E$ we denote $\lambda(K)^{-1}:=\{\lambda_u\in\auto(C^*(E)):u\in K\}$. 
It turns out that $\lambda_u$ is invertible if and only if it is injective and $u^*$ is in the range of 
its unique strictly continuous extension to $M(C^*(E))$, still denoted $\lambda_u$. (Such an extension 
exists since $\lambda_u$ fixes each vertex projection and thus it fixes an approximate unit comprised 
of finite sums of vertex projections.) 
Indeed, if $\lambda_u$ is injective and there exists a $w\in M(C^*(E))$ 
such that $\lambda_u(w)=u^*$ then $w$ must belong to $\U_E$ and $\lambda_u\lambda_w=
\lambda_{u*w}=\id$. Thus $\lambda_u$ is surjective and hence an automorphism. In this case, we 
have $\lambda_u^{-1}=\lambda_w$. 

In the present paper, we mainly deal with {\em finite graphs without sinks.} 
If $E$ is such a graph then the mapping $u\mapsto\lambda_u$ establishes 
a bijective correspondence between $\U_E$ and the semigroup of those unital $*$-homomorphisms 
from $C^*(E)$ into itself which fix all the vertex projections $P_v$, $v\in E^0$. Indeed, 
if $\rho$ is such a homomorphism then $u:=\sum_{e\in E^1}\rho(S_e)S_e^*$ belongs to 
$\U_E$ and $\rho=\lambda_u$ (cf. \cite{Z}).  If $u\in\F_E^1\cap\U_E$ then $\lambda_u$ 
is automatically invertible with inverse $\lambda_{u^*}$ and the map 
\begin{equation}\label{quasifree}
\F_E^1\cap\U_E \ni u\mapsto \lambda_u \in\auto(C^*(E))
\end{equation} 
is a group homomorphism with range inside 
the subgroup of {\em quasi-free automorphisms} of $C^*(E)$, see \cite{DS,Z}. 

If $\lambda_u$ is an endomorphism of $C^*(E)$  corresponding to a unitary $u$ in 
the linear span  of $\{S_\mu S_\nu^*:\mu,\nu\in E^*,\;|\mu|=|\nu|\}$ and the identity
(i.e. in the minimal unitization of the algebraic part of the core AF-subalgebra), then 
we call $\lambda_u$ {\em localized}, cf. \cite{CP}. 

Let $E$ be a finite graph without sinks, and let
\begin{equation}\label{shift}
\varphi(x)=\sum_{e\in E^1} S_e x S_e^* 
\end{equation} 
be the usual {\em shift} on $C^*(E)$, \cite{CK}. It is a unital, completely positive map. One can 
easily verify that the shift 
is an injective $*$-homomorphism when restricted to the relative commutant of $\{P_v:v\in E^0\}$. We 
will denote this relative commutant by $B_E$, i.e.  
\begin{equation}\label{be}
B_E:=(\D_E^0)' \cap C^*(E). 
\end{equation}
In particular, $\U_E=\U(B_E)$. We have 
$\varphi(B_E)\subseteq B_E$ and thus $\varphi(\U_E)\subseteq\U_E$. It is also clear that 
$\varphi(\F_E)\subseteq\F_E$ and $\varphi(\D_E)\subseteq\D_E$. For $k\geq 1$ we denote 
\begin{equation}\label{uk}
u_k := u\varphi(u)\cdots\varphi^{k-1}(u),  
\end{equation}
and agree that $u_k^*$ stands for $(u_k)^*$. For each $u\in\U_E$ and 
for any two paths $\mu,\nu\in E^*$ we have 
\begin{equation}\label{uaction}
\lambda_u(S_\mu S_\nu^*)=u_{|\mu|}S_\mu S_\nu^*u_{|\nu|}^*. 
\end{equation}
The above equality is established with help of the identity $S_e u = \varphi(u) S_e$, which holds for 
all $e\in E^1$ because the unitary $u$ commutes with all vertex projections $P_v$, $v\in E^0$, 
by hypothesis. Indeed,
$$ \varphi(u)S_e=\Big(\sum_{f\in E^1}S_f u S_f^*\Big)S_e=S_euS_e^*S_e
   =S_euP_{r(e)}=S_eP_{r(e)}u=S_eu. $$
More generally, if $x\in C^*(E)$ commutes with all vertex projections then 
\begin{equation}\label{commutation}
S_\alpha x=\varphi^{|\alpha|}(x)S_\alpha
\end{equation}
for each finite path $\alpha$. Furthermore, we have 
\begin{equation}\label{inner}
\Ad(u)=\lambda_{u\varphi(u^*)} \;\;\; \text{for}\; u\in\U_E, 
\end{equation}
where $\Ad(u)(x)=uxu^*$, $x\in C^*(E)$. 


\section{The Weyl group}\label{sectionF}

For algebras $A\subseteq B$ (with a common identity element) 
we denote by $\N_B(A)=\{u\in\U(B):uAu^*=A\}$ the normalizer
of $A$ in $B$, and by $A' \cap B=\{b \in B: (\forall a \in A) \; ab=ba\}$ the 
relative commutant of $A$ in $B$. 

\begin{proposition}\label{normalDautos}
Let $E$ be a finite graph without sinks, and let $u\in\U_E$. Then the following hold: 
\begin{description}
\item{\hspace{1mm}(1)} If $u\D_E u^*\subseteq \D_E$ then $\lambda_u(\D_E)\subseteq\D_E$.  
\item{(2)} If $\lambda_u(\D_E)=\D_E$ then $u\D_Eu^*\subseteq\D_E$. 
\end{description} 
\end{proposition}
\begin{proof}
Part (1) is established in the same way as the analogous statement about $\O_n$ in \cite{Cun2}. 

For part (2), first note that 
$\D_E$ is a $C^*$-algebra generated by $\{\varphi^k(S_e S_e^*) : e \in E^1, \;k=0,1,2,...\}$, 
where $\varphi^0=\id$.
In particular, $\D_E$ is generated by $\D_E^1$ and $\varphi(\D_E)$. Now, if $u\in \U_E$ and 
$\lambda_u(\D_E) = \D_E$ then $u \D_E^1 u^* = \lambda_u(\D_E^1) \subseteq \D_E$ and 
$u \varphi(\D_E) u^* = u\varphi(\lambda_u(\D_E))u^* = \lambda_u(\varphi(\D_E))\subseteq\D_E$. 
Hence $u\D_E u^*\subseteq\D_E$. 
\end{proof}

Further note that if $E$ is a finite graph without sinks in which every loop has an exit 
then $\D_E$ is a MASA in $C^*(E)$ and in part (2) of Proposition \ref{normalDautos} we can 
further conclude that $u\D_E u^*=\D_E$, i.e. that $u\in\N_{C^*(E)}(\D_E)$. In that case, 
by virtue of \cite[Theorem 10.1]{HPP}, every $u\in\N_{C^*(E)}(\D_E)$ can 
be uniquely written as $u=dw$, where $d\in\U(\D_E)$ and $w\in\SS_E$. Here $\SS_E$ denotes 
the group of all  unitaries in $\U(C^*(E))$ of the form $\sum S_\alpha S_\beta^*$ (finite sum). 
Therefore, one can show, as in \cite[Section 2]{CS} that the normalizer of the diagonal in $C^*(E)$ is 
a semi-direct product \begin{equation}\label{bignormalizer}\N_{C^*(E)}(\D_E) = \U(\D_E) \rtimes \SS_E. 
\end{equation}

For $C^*$-algebras $A\subseteq B$ (with a common identity), we denote by $\auto(B,A)$ the collection 
of all those automorphisms $\alpha$ of $B$ such that $\alpha(A)=A$, and by $\auto_A(B)$ 
those automorphisms of $B$ which fix $A$ point-wise. Similarly, $\endo_A(B)$ denotes the collection 
of all unital $*$-homomorphisms of $B$ which fix $A$ point-wise. 

\begin{proposition}\label{diagonalunitaries}
Let $E$ be a finite graph without sinks in which every loop has an exit. Then the mapping 
$u\mapsto\lambda_u$ establishes a group isomorphism 
$$ \endo_{\D_E}(C^*(E))= \auto_{\D_E}(C^*(E)) \cong  \U(\D_E). $$

\end{proposition}
\begin{proof}
It follows immediately from formula (\ref{convolution}) that the mapping $u\mapsto\lambda_u$ 
is a group homomorphism from $\U(\D_E)$ into $\auto(C^*(E))$. If $\mu$ is a finite path 
then $\lambda_u(P_\mu)=\Ad(u_{|\mu|})(P_\mu)$ by (\ref{uaction}). Hence, if $u\in\D_E$, then 
$\lambda_u$ fixes $\D_E$ point-wise. Consequently, the mapping $u\mapsto\lambda_u$ is a well-defined 
group homomorphism into $\auto_{\D_E}(C^*(E))$. The map is one-to-one, 
as already noted in Section \ref{sectionNP}. To see that the map is onto $\endo_{\D_E}(C^*(E))$, 
recall (again from Section 2) that every endomorphism fixing the vertex projections 
is of the form $\lambda_w$ for some $w\in\U_E$. Since $\lambda_w$ fixes $\D_E$ point-wise, 
proceeding by induction on $|\mu|$ one shows that $w$ commutes 
with all projections $P_\mu$. Thus $w\in\U(\D_E)$, since $\D_E$ is a MASA in $C^*(E)$ 
by \cite[Theorem 5.2]{HPP}. 
\end{proof}

We note that under the hypothesis of Proposition \ref{diagonalunitaries} the fixed-point 
algebra for the action of $\auto_{\D_E}(C^*(E))$ on $C^*(E)$ equals $\D_E$. Indeed, 
each element in this fixed-point algebra is also fixed by all $\Ad(u)$, $u\in\U(\D_E)$, and thus belongs 
to $\D_E$. 

\smallskip
For the proof of the following Proposition \ref{autonormalizer}, we note the following simple fact. If $E$ is 
a finite graph without sinks in which every loop has an exit then $\D_E$ does not contain minimal 
projections. Indeed, if $p$ is a non-zero projection in $\D_E$ then there exists a path $\alpha$ such that 
$P_\alpha\leq p$. Extend $\alpha$ to a path ending at a loop, and then further to a path $\beta$ which 
ends at a vertex emitting at least two edges, say $e$ and $f$.  Then for the path $\mu=\beta e$ we have 
$P_{\mu}\lneq P_\alpha\leq p$. 

\begin{proposition}\label{autonormalizer}
Let $E$ be a finite graph without sinks in which every loop has an exit. Then  the 
normalizer of $\auto_{\D_E}(C^*(E))$ in $\auto(C^*(E))$ coincides with $\auto(C^*(E),\D_E)$. 
If, in addition, the center of $C^*(E)$ is trivial, then $\auto_{\D_E}(C^*(E))$ is 
a maximal abelian subgroup of $\auto(C^*(E))$.  
\end{proposition}
\begin{proof}
Let $\alpha\in\auto(C^*(E))$ be in the normalizer of $\auto_{\D_E}(C^*(E))$. Thus for each 
$w\in\U(\D_E)$ there is a $u\in\U(\D_E)$ such that $\alpha\lambda_u=\lambda_w\alpha$. Then 
$\alpha(d)=\alpha\lambda_u(d)=\lambda_w\alpha(d)$ for all $d\in\D_E$. 
Whence $\alpha(\D_E)\subseteq\D_E$. Replacing $\alpha$ with $\alpha^{-1}$ we get the 
reverse inclusion, and thus $\alpha(\D_E)=\D_E$. This proves the first part of the proposition. 

Now let $\alpha\in\auto(C^*(E))$ commute with all elements of $\auto_{\D_E}(C^*(E))$. Then, 
in particular, $\Ad(u)\alpha(x)=\alpha\Ad(u)(x)$ for all $u\in\U(\D_E)$ and $x\in C^*(E)$. Thus 
$u^*\alpha(u)$ belongs to the center of $C^*(E)$ and hence it is a scalar. Therefore 
for each projection $p\in\D_E$ we have $\alpha(2p-1)=
\pm(2p-1)$, and hence $\alpha(p)$ equals either $p$ or $1-p$. 
The case $\alpha(p)=1-p$ is impossible, since taking a projection $0\neq q\lneq p$ we would get 
$\alpha(q)\leq 1-p $ and thus $\alpha(q)\not\in\{q,1-q\}$. Hence $\alpha$ fixes all projections, and 
thus all elements of $\D_E$. The claim now follows from Proposition \ref{diagonalunitaries}. 
\end{proof}

The quotient of $\auto(C^*(E),\D_E)$ by $\auto_{\D_E}(C^*(E))$ will be called the 
{\em Weyl group} of $C^*(E)$ (cf. \cite{Cun2}), and denoted ${\mathfrak W}_E$. That is,  
\begin{equation}\label{weyle}
{\mathfrak W}_E := \auto(C^*(E),\D_E)/\auto_{\D_E}(C^*(E)).
\end{equation}

\begin{remark}\label{zenter}\rm 
Since triviality of the center of $C^*(E)$ plays a role in our considerations, it is worthwile 
to mention the following. If $E$ is a finite graph without sinks in which every loop has an exit then 
the following conditions are equivalent. 
\begin{description}
\item{(i)} The center of $C^*(E)$ is trivial. 
\item{(ii)} $C^*(E)$ is indecomposable. That is, there is no decomposition 
$C^*(E)\cong A\oplus B$ into direct sum of two non-zero $C^*$-algebras. 
\item{(iii)} There are no two non-empty, disjoint, hereditary and saturated subsets $X$ and $Y$ of $E^0$ 
such that: for each vertex $v\in E^0\setminus (X\cup Y)$ there is no loop through $v$, and there exist 
paths from $v$ to both $X$ and $Y$. 
\end{description}
Indeed, it is shown in \cite{H} that conditions (ii) and (iii) above are equivalent. Obviously, (i) implies (ii). To  
see that (ii) implies (i), we first note that in the present case $\D_E$ is a MASA, \cite{HPP}, and thus it 
contains the center $\Z(C^*(E))$ of $C^*(E)$. Thus, the (closed, two-sided) ideal of $C^*(E)$ generated 
by any element of $\Z(C^*(E))$ is gauge-invariant. But for a finite graph $E$, $C^*(E)$ contains only finitely 
many gauge-invariant ideals, \cite{HR} (see also \cite{BHRS} for the complete description of gauge-invariant 
ideals for arbitrary graphs). Hence $\Z(C^*(E))$ is finite dimensional. Thus if $\Z(C^*(E))$ is non-trivial then it 
contains a non-trivial projection and consequently $C^*(E)$ is decomposable. 
\end{remark}

Just as in the case of the Cuntz algebras, the Weyl group of a graph algebra turns out to be countable. 

\begin{proposition}\label{countableweyl}
Let $E$ be a finite graph. Then the Weyl group ${\mathfrak W}_E$  is countable. 
\end{proposition}
\begin{proof}
For each coset in the quotient $\auto(C^*(E),\D_E)/\auto_{\D_E}(C^*(E))$  
choose a representative $\alpha$ and define a mapping from $\auto(C^*(E),\D_E)/\auto_{\D_E}(C^*(E))$ 
to $\oplus^{|E^1|+|E^0|}C^*(E)$ by $\alpha\mapsto\oplus_e\alpha(S_e)\oplus_v\alpha(P_v)$. 
This mapping is one-to-one and the target space is separable. Thus it suffices to show that its image 
is a discrete subset of $\oplus^{|E^1|+|E^0|}C^*(E)$. 

Let $\alpha\in\auto(C^*(E),\D_E)$ be such that $||\alpha(x)-x||<1/2$ for all $x\in\{S_e:e\in E^1\}
\cup\{P_v:v\in E^0\}$. We claim that $\alpha|_{\D_E}=\id$. To this end, we show by induction 
on $|\mu|$ that $\alpha(P_\mu)=P_\mu$ for each path $\mu$. Indeed, if $v\in E^0$ then 
$||\alpha(P_v)-P_v||<1/2$ and thus $\alpha(P_v)=P_v$ since $\alpha(P_v)\in\D_E$ and $\D_E$ is 
commutative. This establishes the base for induction. Now suppose that $\alpha(P_\mu)=\P_\mu$ for 
all paths $\mu$ of length $k$. Let $(e,\mu)$ be a path of length $k+1$. Then, by the 
inductive hypothesis, we have 
$$ ||\alpha(P_{(e,\mu)})-P_{(e,\mu)}||=||\alpha(S_e)P_\mu\alpha(S_e^*)-S_eP_\mu S_e^*||$$
$$ \leq ||\alpha(S_e)P_\mu\alpha(S_e^*)-\alpha(S_e)P_\mu S_e^*||+
   ||\alpha(S_e)P_\mu S_e^*-S_e P_\mu S_e^*||<\frac{1}{2}+\frac{1}{2}=1. $$
Thus $\alpha(P_{(e,\mu)})=P_{(e,\mu)}$, since $\alpha(P_{(e,\mu)})$ and $P_{(e,\mu)}$ commute. 
This yields the inductive step. 

Now suppose that $\alpha,\beta\in\auto(C^*(E),\D_E)$ are such that $||\alpha(x)-\beta(x)||<1/2$ for all 
$x\in\{S_e:e\in E^1\}\cup\{P_v:v\in E^0\}$. Then also $||\beta^{-1}\alpha(x)-x||<1/2$ for 
all such $x$, and hence $\beta^{-1}\alpha\in\auto_{\D_E}(C^*(E))$ by the preceding paragraph. 
This completes the proof. 
\end{proof}

For the remainder of this section, we assume that $E$ is a finite graph without sinks in which every 
loop has an exit. Clearly, each automorphism $\alpha$ of the graph $E$ gives rise to an automorphism 
of the algebra $C^*(E)$, still denoted $\alpha$, such that $\alpha(S_e)=S_{\alpha(e)}$, $e\in E^1$, 
and $\alpha(P_v)=P_{\alpha(v)}$, $v\in E^0$. With a slight abuse of notation, we are identifying 
automorphisms of graph $E$ with the corresponding automorphisms of $C^*$-algebra $C^*(E)$, 
and denote this group $\Gamma_E$. 

We denote by $\gg$ the subgroup of $\auto(C^*(E))$ generated by 
automorphisms of the graph $E$ and $\lambda(\SS_E\cap\U_E)^{-1}$. That is, 
\begin{equation}\label{gggroup}
\gg := \langle \lambda(\SS_E\cap\U_E)^{-1} \cup \Gamma_E \rangle \subseteq \auto(C^*(E)). 
\end{equation}
If $u\in\SS_E\cap\U_E$ then $\lambda_u(\D_E)\subseteq\D_E$. Consequently, if $u\in\SS_E
\cap\U_E$ and $\lambda_u$ is invertible then $\lambda_u$ belongs to $\auto(C^*(E),\D_E)$, since 
$\D_E$ is a MASA in $C^*(E)$. Since each graph automorphism gives rise to an element of 
$\auto(C^*(E),\D_E)$ as well, we have $\gg\subseteq\auto(C^*(E),\D_E)$. We would like to stress that 
the class of automorphisms comprising $\gg$ can be viewed as `purely combinatorial' transformations, which 
facilitates their algorithmic analysis. 

Now, let $u\in\SS_E\cap\U_E$ be such that $\lambda_u$ is invertible. Then, as noted in Section 2 above, 
$\lambda_u^{-1}=\lambda_w$ for some $w\in\U_E$ and $w\in\N_{C^*(E)}(\D_E)$ by 
Proposition \ref{normalDautos}. Let $w=dz$ with $d\in\U(\D_E)$ and $z\in\SS_E$. Then 
we have $z\in\SS_E\cap\U_E$. Since $\lambda_{u*dz}=\lambda_u \lambda_{dz} = \id = \lambda_1$, 
we have $\lambda_u(d)\lambda_u(z)u=u*dz=1$. Thus $d=1$, by (\ref{bignormalizer}), 
and $\lambda_u^{-1}=\lambda_z$. This shows that $\lambda(\SS_E\cap\U_E)^{-1}$ is a group. 
If $\alpha$ is an automorphism of $E$ and $u\in\SS_E\cap\U_E$ then $\alpha \lambda_u \alpha^{-1}=
\lambda_w$ for some $w\in\U_E$, since $\alpha \lambda_u \alpha^{-1}$ fixes all the vertex projections. 
A short calculation shows that this $w$ belongs to $\SS_E$. It follows that 
$\lambda(\SS_E\cap\U_E)^{-1}$ is a normal subgroup of $\gg$. Let $\Gamma_E^0$ 
be the normal subgroup of $\Gamma_E$ consisting of those automorphisms which fix 
all vertex projections.  It is easy to verify that $\Gamma_E^0=\lambda(\P_E^1\cap\U_E)$, where 
we denote $\P_E^1=\F_E^1\cap\SS_E$. We have 
\begin{equation}\label{ggstructure}
\gg = \lambda(\SS_E\cap\U_E)^{-1} \Gamma_E \;\;\; \text{and} \;\;\; 
\lambda(\SS_E\cap\U_E)^{-1} \cap \Gamma_E = \Gamma_E^0. 
\end{equation}

\begin{proposition}\label{weylembedding} 
Let $E$ be a finite graph without sinks in which every loop has an exit. Then 
there is a natural embedding of $\gg$ into the Weyl group ${\mathfrak W}_E$ of $C^*(E)$. 
\end{proposition}
\begin{proof}
Since $\gg\subseteq \auto(C^*(E),\D_E)$, it suffices to show that 
$\gg \cap \auto_{\D_E}(C^*(E)) = \{\id\}$. Indeed, if 
$\beta\in \gg $ then $\beta=\lambda_u \alpha$ for some 
$u\in\SS_E\cap\U_E$ and $\alpha\in\auto(E)$. If, in addition, $\beta\in\auto_{\D_E}(C^*(E))$ then $\alpha\in
\Gamma_E^0$ and thus $\alpha=\lambda_w$ for some $w\in\P_E^1\cap\U_E$. Therefore 
$\beta=\lambda_u\lambda_w=\lambda_{u*w}$ and $u*w\in\SS_E$. By Proposition \ref{diagonalunitaries} 
we also have $\beta=\lambda_d$ for some $d\in\U(\D_E)$. Consequently $u*w=d=1$, since 
$\SS_E\cap\U(\D_E)=\{1\}$, and thus $\beta=\id$. 
\end{proof}

One of the more difficult issues arising in dealing with automorphisms of graph $C^*$-algebras is 
deciding if the automorphism at hand is outer or inner. The following theorem provides a criterion 
useful for certain automorphisms belonging to the subgroup $\gg$ of the Weyl group 
${\mathfrak W}_E$, see Corollary \ref{outerness} below. 

\begin{theorem}\label{fullinner}
Let $E$ be a finite graph without sinks in which every loop has an exit. Then 
$$ \gg\cap\inn(C^*(E))\subseteq\{\Ad(w):w\in\SS_E\}. $$
\end{theorem}
\begin{proof}
By (\ref{ggstructure}), each element of $\gg$ is of the form $\lambda_u \alpha$, with $u\in\SS_E\cap\U_E$ 
and $\alpha\in\Gamma_E$. If such an automorphism is inner and equals $\Ad(y)$ for some 
$y\in\U(C^*(E))$ then $y\in\N_{C^*(E)}(\D_E)$, 
and thus we have $\lambda_u \alpha = \Ad(d) \Ad(w)$ for some $d\in\U(\D_E)$ and $w\in\SS_E$,  
by (\ref{bignormalizer}). Hence $\rho:=\lambda_u \alpha \Ad(w^*)=\lambda_{d\varphi(d^*)}$. But 
then $d\varphi(d^*)=\sum_{e\in E^1}\rho(S_e)S_e^*\in\SS_E$.  Since $\SS_E\cap\U(\D_E)=\{1\}$, we 
have $d\varphi(d^*)=1$. Therefore $d=\varphi(d)$ and this implies (via a straightforward calculation) that 
$d$ belongs to the center of $C^*(E)$. Hence $\Ad(u)=\id$ and thus $\lambda_u \alpha= \Ad(w)$, 
as required. 
\end{proof}


\section{The restricted Weyl group}\label{sectionR}

In order to define and study the restricted Weyl group of a graph $C^*$-algebra, 
we need some preparation on endomorphisms which globally preserve its 
core AF-subalgebra. Recall that $\gamma:U(1)\to\auto(C^*(E))$ is the gauge action, such that $\gamma_z(S_e)=zS_e$ for all $e\in E^1$ and $z\in U(1)$. Then $\F_E$ is the fixed-point 
algebra for action $\gamma$. 

\medskip
The same argument as in Proposition \ref{normalDautos} yields the following. 

\begin{proposition}\label{normalFautos}
Let $E$ be a finite graph without sinks, and let $u\in\U_E$. Then the following hold: 
\begin{description}
\item{(1)} If $u\F_E u^*\subseteq \F_E$ then $\lambda_u(\F_E)\subseteq\F_E$. 
\item{(2)} If $\lambda_u(\F_E) = \F_E$ then $u\F_Eu^*\subseteq\F_E$. 
\end{description} 
\end{proposition}

It turns out that in many instances the normalizer of the core AF-subalgebra $\F_E$ of $C^*(E)$ 
is trivial. The following theorem is essentially due to Mikael R{\o}rdam, \cite{Ror2}. 

\begin{theorem}\label{normalizer}
Let $E$ be a directed graph with finitely many vertices and no sources. Suppose 
further that the relative commutant of $\F_E$ in $C^*(E)$ is trivial. Then  
$u\in\U(C^*(E))$ and $u\F_E u^*\subseteq \F_E$ imply that $u\in\U(\F_E)$. 
\end{theorem}
\begin{proof}
If $x\in\F_E$ then $uxu^*\in\F_E$ and thus $\gamma_t(u)x\gamma_t(u)^*=
\gamma_t(uxu^*)=uxu^*$ for each $t\in U(1)$. Consequently, 
$u^*\gamma_t(u)$ belongs to $\F_E'\cap C^*(E)=\bC1$. Thus, for each $t\in U(1)$ there 
exists a $z(t)\in\bC$ such that $\gamma_t(u)=z(t)u$. Clearly, $t\mapsto z(t)$ is 
a continuous character of the circle group $U(1)$. Hence there is an integer $m$ such 
that $z(t)=t^m$. If $m=0$ then $u$ is invariant under the gauge action and we are done. 
Otherwise, by passing to $u^*$ if necessary, we may assume that $m>0$.  

For each vertex $v\in E^0$ choose one edge $e_v$ with range $v$ and let $T= \sum_{v\in E^0} S_{e_v}$. 
We have $\gamma_t(T)=tT$ for all $t\in U(1)$ and $T$ is an isometry, since $E$ has no sources. 
Furthermore, $TT^*\neq 1$, for otherwise each vertex would emit exactly one edge. As $E^0$ is finite 
and there are no sources, this would entail existence of a loop disjoint from the rest of the graph, and 
consequently $\F_E'\cap C^*(E)$ would contain the non-trivial center of $C^*(E)$, contrary to the assumptions.  
Now $T^m u^*$ is fixed by the gauge action and thus it is an isometry in $\F_E$. Since $\F_E$ is an 
AF-algebra, $T^m u^*$ must be unitary. But then $T^m$ and thus $T$ itself would 
be unitary, a contradiction. This completes the proof.  
\end{proof}

\begin{remark}\label{variouscenters}\rm
Several results of this paper depend on triviality  of the center $\Z(C^*(E))$ of the graph algebra $C^*(E)$, 
the center $\Z(\F_E)$ of the core AF-subalgebra $\F_E$, or the relative commutant $\F_E'\cap C^*(E)$, 
respectively. The most interesting case to us is when $E$ is  finite without sinks and where all loops have exits. 
Since in this case $\D_E$ is maximal abelian in $C^*(E)$, \cite{HPP}, it follows immediately that  
$$ \Z(C^*(E)) \subseteq \Z(\F_E) = \F_E'\cap C^*(E). $$ 
The first inclusion above may be strict, even if $C^*(E)$ is simple (see \cite{PRho} for examples and discussion). 
However, when $E$ is strongly connected (i.e. for every pair of vertices $v$, $w$ there is a path from $v$ 
to $w$) and has period 1 (i.e. the greatest common divisor of the lengths of all loops is 1) 
then $\F_E$ is simple and thus its center is trivial, \cite[Theorem 6.11]{PRho}. 
\end{remark}

We begin our analysis of automorphisms of $C^*(E)$ which globally preserve $\F_E$ by identifying 
those which fix $\F_E$ point-wise.  

\begin{proposition}\label{fixingf}
Let $E$ be a finite graph without sinks in which every loop has an exit. Suppose also that the 
center of $\F_E$ is trivial. Then 
$$ \endo_{\F_E}(C^*(E))=\auto_{\F_E}(C^*(E))=\{\gamma_z:z\in U(1)\}. $$
\end{proposition}
\begin{proof}
Let $\alpha\in\endo_{\F_E}(C^*(E))$. Then $\alpha|_{\D_E}=\id$ and thus 
$\alpha=\lambda_u$ for some $u\in\U(\D_E)$, as noted in Section 2. Then an 
easy induction on $k$ shows that $u$ commutes with all $\F_E^k$, and thus $u\in\F_E'\cap\D_E$. 
But $\F_E'\cap\D_E$ is the center of $\F_E$, since $\D_E$ is a MASA in $\F_E$. Hence $u$ is 
a scalar and $\alpha=\lambda_u$ is a gauge automorphism. 
\end{proof}

In what follows, we will be mainly working with a finite graph $E$ without sinks and sources. 
For each vertex $v\in E^0$ we select exactly one edge $e_v$ with $r(e_v)=v$, and define 
\begin{equation}\label{T}
T=\sum_{v\in E^0}S_{e_v},  
\end{equation}
as in the proof of Theorem \ref{normalizer}. 
Then $T$ is an isometry such that $\gamma_z(T)=zT$ for all $z\in U(1)$. 

\begin{lemma}\label{fautos}
Let $E$ be a finite graph without sinks and sources in which every loop has an exit. Also, we assume 
that the center of $\F_E$ is trivial. For an automorphism $\alpha\in\auto(C^*(E))$ the following 
conditions are equivalent.
\begin{description}
\item{(1)} $\alpha\gamma_z=\gamma_z\alpha$ for each $z\in U(1)$;
\item{(2)} For each $e\in E^1$ there is a $w\in\F_E$ such that $\alpha(S_e)=wT$, where $T$ is 
an isometry as in (\ref{T});  
\item{(3)} $\alpha(\F_E)=\F_E$.
\end{description}
\end{lemma}
\begin{proof}
(1)$\Rightarrow$(3): If $\alpha$ commutes with each $\gamma_z$ then $\alpha(\F_E)\subseteq\F_E$, 
since $\F_E$ is the fixed-point algebra for the gauge action. Also, $\alpha\gamma_z=\gamma_z\alpha$ 
implies $\alpha^{-1}\gamma_z=\gamma_z\alpha^{-1}$, and thus $\alpha^{-1}(\F_E)\subseteq\F_E$ 
as well. 

(3)$\Rightarrow$(1): Suppose $\alpha(\F_E)=\F_E$. Then for each $x\in\F_E$ and $z\in U(1)$ we have 
$\alpha\gamma_z\alpha^{-1}\gamma_z^{-1}(x)=x$. Then by Proposition \ref{fixingf} there exists 
an $\omega(z)\in U(1)$ such that $\alpha\gamma_z\alpha^{-1}\gamma_z^{-1}(x)=\gamma_{\omega(z)}$. 
The mapping $z\mapsto \omega(z)$ is a continuous character of $U(1)$. Indeed, 
\begin{align*}
\gamma_{\omega(zy)} & = \alpha\gamma_{zy}\alpha^{-1}\gamma_{zy}^{-1} \\ 
 & = (\alpha\gamma_z\alpha^{-1}\gamma_z^{-1})\gamma_z(\alpha\gamma_y\alpha^{-1}\gamma_y^{-1})
\gamma_y\gamma_{zy}^{-1} \\ 
 & = \gamma_{\omega(z)}\gamma_z\gamma_{\omega(y)}\gamma_y\gamma_{zy}^{-1} \\ 
 & = \gamma_{\omega(z)}\gamma_{\omega(y)}. 
\end{align*}
Hence there exists an $m\in\bZ$ such that $\omega(z)=z^m$ and, consequently, we have 
$\alpha\gamma_z=\gamma_{z^{m+1}}\alpha$ for all $z\in U(1)$. Therefore $\gamma_z(\alpha^{-1}(S_e))
=z^{m+1}\alpha^{-1}(S_e)$ for all $e\in E^1$. Since $\alpha^{-1}$ is an automorphism, this 
implies that $C^*(E)$ is generated by the spectral subspace $C^*(E)^{(m+1)}$ of the gauge action. Thus 
$m+1=\pm 1$. However, the case $m+1=-1$ is impossible. Indeed, let $T$ be an isometry as in 
(\ref{T}). Since $1=\sum_{e\in E^1}S_e S_e^*$, we have $\alpha^{-1}(T)=\sum_{e\in E^1}x_e S_e^*$, 
where $x_e=\alpha^{-1}(T)S_e$ are partial isometries in $\F_E$ such that 
$x_e^*x_e=S_e^*S_e=P_{r(e)}$. Thus 
$$ \sum_{e\in E^1}x_e^*x_e=\sum_{e\in E^1}P_{r(e)}>1, $$
due to our assumptions on the graph $E$. On the other hand, 
$$ \sum_{e\in E^1}x_ex_e^*=\sum_{e\in E^1}\alpha^{-1}(T)S_eS_e^*\alpha^{-1}(T^*)\leq 
\alpha^{-1}(TT^*)\leq 1. $$
This yields a contradiction $ \sum_{e\in E^1}x_e^*x_e > \sum_{e\in E^1}x_ex_e^*$, since $\F_E$ 
being an AF-algebra is stably finite. 

(1)$\Rightarrow$(2): We have $\alpha(S_e)=\alpha(S_e)T^*T=\sum_{v\in E^0}\alpha(S_e)S_{e_v}^*S_{e_v}$. 
If $\alpha$ commutes with the gauge action then $\alpha(S_e)S_{e_v}^*$ belongs to $\F_E$, 
and the claim follows. 

(2)$\Rightarrow$(1): This is clear. 
\end{proof}

Now, we turn to the definition of the restricted Weyl group of $C^*(E)$. 
If $E$ has no sinks and all loops have exits, 
then each $\alpha\in\auto_{\D_E}(C^*(E))$ automatically belongs to $\auto(C^*(E),\F_E)$ by 
Proposition \ref{diagonalunitaries}, above. Thus we may consider the quotient of 
$$ \auto(C^*(E),\F_E,\D_E) := \auto(C^*(E),\D_E)\cap\auto(C^*(E),\F_E) $$ 
by $\auto_{\D_E}(C^*(E))$, which will be called the 
{\em restricted Weyl group} of $C^*(E)$ and denoted ${\mathfrak R}{\mathfrak W}_E$. That is, 
\begin{equation}\label{rweyle}
{\mathfrak R}{\mathfrak W}_E := \auto(C^*(E),\F_E,\D_E)/\auto_{\D_E}(C^*(E)). 
\end{equation}

We denote by $\P_E^k$ the collection of all unitaries in $\U(C^*(E))$ of the form 
$\sum S_\alpha S_\beta^*$ with $|\alpha|=|\beta|=k$. Clearly, each $\P_E^k$ is a finite subgroup 
of $\U(C^*(E))$, and we have $\P_E^k\subseteq\P_E^{k+1}$ for each $k$. We set $\P_E:=
\bigcup_{k=0}^\infty\P_E^k$. As shown in \cite{P}, every $u\in\N_{\F_E}(\D_E)$ can be uniquely 
written as $u=dw$, where $d\in\U(\D_E)$ and $w\in\P_E$. That is, the normalizer of 
the diagonal in $\F_E$ is a semi-direct product 
\begin{equation}\label{smallnormalizer}
\N_{\F_E}(\D_E) = \U(\D_E) \rtimes \P_E. 
\end{equation}

We denote by $\g$ the subgroup of $\auto(C^*(E))$ 
generated by automorphisms of the graph $E$ and $\lambda(\P_E\cap\U_E)^{-1}$. That is, 
\begin{equation}\label{ggroup}
\g := \langle \lambda(\P_E\cap\U_E)^{-1} \cup \Gamma_E \rangle \subseteq \auto(C^*(E)). 
\end{equation}
By Proposition \ref{weylembedding}, there is a natural embedding of $\gg$ into the Weyl group 
${\mathfrak W}_E$ of $C^*(E)$. Its restriction yields an embedding of $\g$ into the 
restricted Weyl group ${\mathfrak R}{\mathfrak W}_E$ of $C^*(E)$. Furthermore, 
$\lambda(\P_E\cap\U_E)^{-1}$ is a normal subgroup of $\g$. Hence we have 
\begin{equation}\label{gstructure}
\g = \lambda(\P_E\cap\U_E)^{-1} \Gamma_E \;\;\; \text{and} \;\;\; 
\lambda(\P_E\cap\U_E)^{-1} \cap \Gamma_E = \Gamma_E^0. 
\end{equation}

Similarly to Theorem \ref{fullinner}, in the restricted case we have the following. 

\begin{proposition}\label{restrictedinner}
Let $E$ be a finite graph without sinks and sources in which every loop has an exit, and such that 
the relative commutant of $\F_E$ in $C^*(E)$ is trivial. Then 
$$ \g\cap\inn(C^*(E))\subseteq\{\Ad(w):w\in\P_E\}. $$
\end{proposition}
\begin{proof}
By Theorem \ref{fullinner}, every element of $\g\cap\inn(C^*(E))$ is of the form $\Ad(w)$ with 
$w\in\SS_E$. Since this $\Ad(w)$ globally preserves $\F_E$, by hypothesis, Theorem  
\ref{normalizer} implies that $w\in\P_E$. 
\end{proof}

Since for each $w\in\P_E$ the corresponding inner automorphism $\Ad(w)$ of $C^*(E)$ has finite order, 
Proposition \ref{restrictedinner} immediately yields the following. 

\begin{corollary}\label{outerness}
Let $E$ be a finite graph without sinks and sources in which every loop has an exit, and such that 
the relative commutant of $\F_E$ in $C^*(E)$ is trivial. Then every element of infinite order in 
$\g$ has infinite order in $\outo(C^*(E))$ as well. 
\end{corollary}

In general, it is a non-trivial matter to verify outerness of an automorphism of a graph algebra. 
Corollary \ref{outerness} solves this problem for a significant class of automorphisms. In the case 
of Cuntz algebras, an analogous result was proved in \cite[Theorem 6]{Szy}, and provided a convenient 
outerness criterion for permutative automorphisms -- probably, the most studied class 
of automorphisms of $\O_n$. 

\medskip
Now, we turn to analysis of the action of the restricted Weyl group on the diagonal MASA. 
Let $E$ be a finite graph without sinks in which every loop has an exit. If $\alpha$ is an automorphism of 
$\D_E$ then, following \cite{CHS1}, we say that $\alpha$ has property (P) if there exists a
non-negative integer $m$ such 
that the endomorphism $\alpha\varphi^m$ commutes with the shift $\varphi$. In that case, we also 
say that $\alpha$ satisfies (P) with $m$. We define 
\begin{equation}\label{propertyP}
\aee:=\{\alpha\in\auto(\D_E):\text{ both $\alpha$ and $\alpha^{-1}$ have property (P)}\}. 
\end{equation}
Clearly, $\aee$ is a subgroup of $\auto(\D_E)$. In the case of the Cuntz algebras, it was shown in 
\cite{CHS1} that the restricted Weyl group is naturally isomorphic to $\aee$. Now, 
we investigate this problem for graph algebras. 

\begin{lemma}\label{dfautos}
Let $E$ be a finite graph without sinks and sources in which every loop has an exit. Also, 
we assume that the center of $\F_E$ is trivial. 
If $\alpha\in\auto(C^*(E),\D_E)\cap\auto(C^*(E),\F_E)$ then for each $e\in E^1$ 
there exists a partial unitary $d\in\D_E$ and a partial isometry $w\in\F_E$ equal to a  
finite sum of words of the form $S_\mu S_\nu^*$ with $|\mu|=|\nu|$, and such that 
\begin{equation}\label{formdfauto}
\alpha(S_e)=dwT,   
\end{equation}
where $T$ is an isometry as in (\ref{T}).   
\end{lemma}
\begin{proof}
For each $e\in E^1$, $S_e$ is a 
partial isometry normalizing $\D_E$ (i.e., $S_e\D_E S_e^*\subseteq\D_E$ and $S_e^*\D_E S_e 
\subseteq\D_E$). Since $\alpha(\D_E)=\D_E$, it follows that $\alpha(S_e)$ normalizes $\D_E$ 
as well. By \cite[Theorem 10.1]{HPP}, $\alpha(S_e)$ equals $dv$ for some partial unitary $d\in\D_E$ 
and $v$ a partial isometry which can be written as the sum of a finite collection of words 
$S_\mu S_\nu^*$. Since $\alpha$ commutes with 
the gauge action by Lemma \ref{fautos}, we have $\gamma_z(v)=zv$ for each $z\in U(1)$. 
Thus for each of the words in the decomposition of $v$ we have $|\mu|=|\nu|+1$. 
Now let $T$ be as in (\ref{T}). Then $\alpha(S_e)=dv=dwT$ 
is the desired decomposition, with $w=vT^*$. 
\end{proof}

\begin{proposition}\label{alphaaction}
Let $E$ be a finite graph without sinks and sources in which every loop has an exit. Also, 
we assume that the center of $\F_E$ is trivial. Let $\alpha\in\auto(C^*(E),\D_E)\cap\auto(C^*(E),\F_E)$ 
and let $T$ be an isometry as in (\ref{T}). For each $e\in E^1$ let $d_e$ and $w_e$ be as in 
Lemma \ref{dfautos} so that $\alpha(S_e)=d_ew_eT$. Then for each path $\mu$ of length $r$ we have 
$$ \alpha(S_\mu)=D_\mu W_\mu T^r, $$
where $D_\mu\in\D_E$  and $W_\mu$, a sum of words in $\F_E$, are such that 
\begin{align*}
D_\mu & = (d_{\mu_1}w_{\mu_1}T)(d_{\mu_2}w_{\mu_2}T)\cdots(d_{\mu_{r-1}}w_{\mu_{r-1}}T)
d_{\mu_r}(w_{\mu_{r-1}}T)^*\cdots(w_{\mu_1}T)^*, \\
W_\mu & = (w_{\mu_1}T)(w_{\mu_2}T)\cdots(w_{\mu_{r-1}}T)w_{\mu_r}T^{*(r-1)}. 
\end{align*}
Furthermore, $D_\mu$ is a partial unitary and $W_\mu$ is a partial isometry. Thus we also have 
$$ \alpha(P_\mu)=D_\mu D_\mu^*W_\mu T^rT^{*r}W_\mu^*. $$
\end{proposition}
\begin{proof}
This is established by a somewhat tedious but not complicated inductive argument. We illustrate it with 
the case $r=2$ only. Since $T^*T=1$ and $w_{\mu_1}$ is a partial isometry, we have 
$$ \alpha(S_{\mu_1}S_{\mu_2})=(d_{\mu_1}w_{\mu_1}T)(d_{\mu_2}w_{\mu_2}T)=
d_{\mu_1}w_{\mu_1}(w_{\mu_1}^*w_{\mu_1})(Td_{\mu_2}T^*)Tw_{\mu_2}T. $$
But both $w_{\mu_1}^*w_{\mu_1}$ and $Td_{\mu_2}T^*$ belong to $\D_E$ (since $w_{\mu_1}$ and 
$T$ are partial isometries normalizing $\D_E$). Thus the above expression equals 
$$ [(d_{\mu_1}w_{\mu_1}T)d_{\mu_2}(w_{\mu_1}T)^*][(w_{\mu_1}T)w_{\mu_2}T^*]T^2, $$ 
as required. Note that $D_\mu=(d_{\mu_1}w_{\mu_1}T)d_{\mu_2}(w_{\mu_1}T)^*=
d_{\mu_1}(w_{\mu_1}Td_{\mu_2}T^*w_{\mu_1}^*)$ is a partial unitary in $\D_E$, and 
$W_\mu=(w_{\mu_1}T)w_{\mu_2}T^*$ is a sum of words in $\F_E$ and a partial isometry 
(as a product of partial isometries with mutually commuting domain and range projections). 
\end{proof}

\begin{corollary}\label{dimensioncount}
Keeping the notation and hypothesis from Proposition \ref{alphaaction}, let 
$$ M=min\{k:(\forall e\in E^1)\,d_e\in\D_E^k,\,w_e\in\F_E^k\}. $$ 
Then $D_\mu\in\D_E^{M+|\mu|-1}$ and $W_\mu\in\F_E^{M+|\mu|-1}$. 
\end{corollary}

\begin{remark}\label{alphaactionremark}
\rm In fact, the conclusion of Proposition \ref{alphaaction} remains valid with no hypothesis on the graph 
$E$ and for any endomorphism $\alpha$ of $C^*(E)$, if we know that $\alpha(S_e)=d_ew_eT$ 
for each $e\in E^1$. 
\end{remark}

\begin{remark}\label{pn1gauge}
\rm Let $T=\sum_{v\in E^0}S_{e_v}$ and $T'=\sum_{v\in E^0}S_{f_v}$ be two isometries as in 
(\ref{T}). Set $\widetilde{U}=\sum_{v\in E^0}S_{f_v}S_{e_v}^*$. Then $\widetilde{U}$ is a 
partial isometry in the finite dimensional $C^*$-algebra $\F_E^1$. We may extend it to a unitary $U\in
\P_E^1$, and then we have $T'=UT$. Thus if $\alpha\in\auto(C^*(E),\F_E)\cap\auto(C^*(E),\D_E)$ and 
$\alpha(S_e)=d_ew_eT=d_e'w_e'T'$, as in Lemma \ref{dfautos}, then $d_e'=d_e$ and $w_e'U=w_e$. 
In particular, if we fix $T$, then for each $e\in E^1$ there is a unique such $w_e$ which satisfies 
$\alpha(S_e)=d_e w_e T$ and $w_e=\alpha(S_e S_e^*)w_e TT^*$. 
\end{remark}

Let $\alpha$ be an endomorphism of $C^*(E)$, where $E$ is a finite graph without sinks. Define 
a unital, completely positive map $\Phi_\alpha:C^*(E)\to C^*(E)$ by 
\begin{equation}\label{Phialpha}
\Phi_\alpha(x)=\sum_{e\in E^1}\alpha(S_e)x\alpha(S_e^*). 
\end{equation}
Then the following braiding relation holds
\begin{equation}\label{braiding}
\alpha\varphi=\Phi_\alpha\alpha. 
\end{equation}
If $\alpha=\lambda_u$ for some $u\in\U_E$, then $\Phi_\alpha=\Ad(u)\circ\varphi$. 

\medskip
Now, we are in the position to show how elements of the restricted Weyl group act on 
the diagonal, by automorphisms of $\D_E$ (or, equivalently, homeomorphisms of its spectrum) 
which eventually commute with the shift. 

\begin{theorem}\label{eventuallycommute}
Let $E$ be a finite graph without sinks and sources in which every loop has an exit. Also, 
we assume that the center of $\F_E$ is trivial. 
If $\alpha\in\auto(C^*(E),\D_E)\cap\auto(C^*(E),\F_E)$ then the restriction of $\alpha$ to $\D_E$ 
belongs to $\aee$. This yields a group homomorphism 
$$ \res:\auto(C^*(E),\D_E)\cap\auto(C^*(E),\F_E)\longrightarrow\aee $$
and a group embedding 
$$ {\mathfrak R}{\mathfrak W}_E \hookrightarrow \aee. $$
\end{theorem}
\begin{proof}
It suffices to show that there exists an $m$ such that $\alpha^{-1}\varphi^{m+1}=\varphi\alpha^{-1}
\varphi^m$ on $\D_E$. But $\alpha\varphi\alpha^{-1}\varphi^m=\Phi_\alpha\varphi^m$. Thus, we must 
show that $\Phi_\alpha\varphi^m=\varphi^{m+1}$ on $\D_E$ for a sufficiently large $m$. To this end, 
for each $e\in E^1$ write $\alpha(S_e)=d_ew_eT$ as in Lemma \ref{dfautos}. Let $m$ be so large that 
all $w_e$ belong to $\F_E^{m+1}$. Then, for $x\in\D_E$, using relation (\ref{commutation}) we 
have $T\varphi^m(x)=\varphi^{m+1}(x)T$ and each $w_e$ commutes with $\varphi^{m+1}(x)$. 
Consequently,  we have 
$$ \Phi_\alpha\varphi^m(x)=\sum_{e\in E^1}d_ew_eT\varphi^m(x)T^*w_e^*d_e^*=
\varphi^{m+1}(x)\sum_{e\in E^1}d_ew_eTT^*w_e^*d_e^*=\varphi^{m+1}(x), $$
as required. 

Finally, the kernel of the $\res$ homomorphism coincides with 
$\auto_{\D_E}(C^*(E))$. Thus, the restriction gives rise to a homomorphic embedding of the 
restricted Weyl group ${\mathfrak R}{\mathfrak W}_E$ into $\aee$. 
\end{proof}

In the case of $\O_n$, it was shown in \cite{CHS1} that the restriction mapping from Theorem 
\ref{eventuallycommute} is surjective. In this way, the restricted Weyl group of the Cuntz algebra 
has been identified with the group of homeomorphisms which (along with their inverses) eventually 
commute with the full one-sided $n$-shift. The more general case of graph algebras is more complicated.  
We would like to pose it as an open problem to determine the exact class 
of automorphisms of $\D_E$ (or, equivalently, homeomorphisms of the underlying space) which 
arise as restrictions of automorphisms of the graph algebra which preserve both the diagonal and 
the core AF-subalgebra.


\section{The localized automorphisms}\label{sectionL}

Throughout this section, we assume that $E$ is a finite graph without sinks. 
Recall that an endomorphism $\lambda_u$ of $C^*(E)$ is called {\em localized} if the corresponding 
unitary $u$ belongs to a finite dimensional algebra $\F_E^k\cap\U_E$ for some $k$. 
Our main aim in this section is to produce an invertibility criterion for localized 
endomorphisms, analogous to \cite[Theorem 3.2]{CS}. 

\medskip
Let $u$ be unitary in $\F_E^k\cap\U_E$, for some fixed $k\geq 1$. Using (\ref{uk}) and (\ref{uaction}) 
we see that $\lambda_u(x)={\rm Ad}(u_r)(x)$ for all $x\in\F_E^r$ and $r\geq 1$. Following \cite{Szy}, 
for each pair $e, f \in E^1$ we define a linear map $a_{e,f}^u: \F_E^{k-1} \to \F_E^{k-1}$ by
\begin{equation}\label{aefu}
a_{e,f}^u(x) = S_e^* u^* x u S_f, \;\;\; x \in \F_E^{k-1}.
\end{equation}
Denote $V_k := \F_E^{k-1}/\D^0_E$, the quotient vector space, and let  ${\mathcal L}(V_k)$ 
be the space of linear maps from $V_k$ to itself. 
Since $a_{e,f}^u(\D^0_E) \subseteq \D^0_E$, there is an induced map
$\tilde{a}_{e,f}^u: V_k \to V_k$. Now we 
define $A_u$ as the subring of ${\mathcal L}(V_k)$ generated by
$\{\tilde{a}_{e,f}^u  \ : \ e, f\in E^1 \}$.

We denote by $H$ the linear span of the generators $S_e$'s. Let $u$ be as above. 
Following \cite[p. 386]{CP}, we define inductively
\begin{equation}
\Xi_0 = \F_E^{k-1}, \quad
\Xi_{r} = \lambda_u(H)^* \Xi_{r-1} \lambda_u(H) \, , \ r \geq 1 \ .
\end{equation}
Then $\{\Xi_r\}_r$ is a non-increasing sequence of finite dimensional, self-adjoint subspaces 
of $\F_E^{k-1}$ and thus 
it eventually stabilizes. If $\Xi_p=\Xi_{p+1}$ then we have $\Xi_u:=\bigcap_{r=0}^\infty\Xi_r=\Xi_p$. 

If $\alpha,\beta$ are paths of length $r$, then we denote by $T_{\alpha,\beta}$ the linear map from 
$\F_E^{k-1}$ to itself defined by $T_{\alpha,\beta}=a_{\alpha_r,\beta_r}^u\cdots a_{\alpha_1,\beta_1}^u$. 
We have $T_{\alpha,\beta}(x)=S_\alpha^*\Ad(u_r^*)(x)S_\beta$ for all $x\in\F_E^{k-1}$. Note 
that $T_{\alpha,\beta}T_{\mu,\nu}=0$ if either $\alpha\mu$ or $\beta\nu$ does not form a path. It easily 
follows from our definitions that the space $\Xi_r$ is linearly spanned by elements of the form 
$T_{\alpha,\beta}(x)$, for $\alpha,\beta\in E^r$, $x\in\F_E^{k-1}$. 

\begin{theorem}\label{main}
Let $E$ be a finite graph without sinks, and let $u\in\U_E$ be a unitary in $\F_E^k$ for some $k \geq 1$. 
Then the following conditions are equivalent:
\begin{itemize}
\item[(1)] $\lambda_u$ is invertible with localized inverse;
\item[(2)] the sequence of unitaries $\{ {\rm Ad}(u_m^*)(u^*)\}_{m \geq 1}$ eventually stabilizes;
\item[(3)] the ring $A_u$ is nilpotent;
\item[(4)] $\Xi_u \subseteq \D_E^0$.
\end{itemize}
\end{theorem}
\begin{proof}
(1) $\Rightarrow$ (3): If the inverse of $\lambda_u$ is localized then 
there exists an $l$ such that $\lambda_u^{-1}(\F_E^{k-1})\subseteq\F_E^l$. 
Let $\alpha=(e_1,e_2,\cdots,e_l)$ and $\beta=(f_1,f_2,\cdots, f_l)$ be paths of length $l$, 
and consider an element $T_{\alpha, \beta}=a_{e_l,f_l}^u\cdots a_{e_2,f_2}^ua_{e_1,f_1}^u$ 
of $A_u^l$. Let $b\in\F_E^{k-1}$ and let $x=\lambda_u^{-1}(b)$. Then $x\in\F_E^l$ and we have 
$b=\lambda_u(x)=\Ad(u_l)(x)$. Therefore
$$ T_{\alpha, \beta}(b)=a_{e_l,f_l}^u\cdots a_{e_2,f_2}^ua_{e_1,f_1}^u(b)=
S_\alpha^*{\rm Ad}(u^*_l)(b)S_\beta=S_\alpha^*xS_\beta. $$
Since $x$ can be written as $\sum_{|\gamma|=|\rho|=l}c_{\gamma,\rho}(x)
S_\gamma S_\rho^*$ for some $c_{\gamma,\rho}(x)\in\bC$, we have
$$ T_{\alpha, \beta}(b)=S^*_\alpha\Big(\sum_{|\gamma|=|\rho|=l}c_{\gamma,\rho}(x)S_\gamma S_\rho^*\Big)S_\beta =\begin{cases} c_{\alpha,\beta}(x)P_{r(\alpha)} , & \; \text{if}\; r(\alpha)=r(\beta)\\
             0 , & \; \text{if}\; r(\alpha)\neq r(\beta)\\
\end{cases} $$
because $S_\alpha^*S_\gamma=P_{r(\alpha)}$ if $\alpha=\gamma$, and $S_\alpha^*S_\gamma=0$ 
otherwise. This implies that $T_{\alpha, \beta}(b)\in\D_E^0$, and hence we see that $A_u^l=0$.

(3) $\Rightarrow$ (4): Let $A_u^l = 0$ for some positive integer $l$. Then $T_{\alpha,\beta}(b)\in\D_E^0$ 
for all $b\in\F_E^{k-1}$ and all $\alpha,\beta$ such that $|\alpha|=|\beta|=l$. But this immediately 
yields $\Xi_l\subseteq\D_E^0$ and, consequently, $\Xi_u\subseteq\D_E^0$. 

(4) $\Rightarrow$ (2): Let $\Xi_u\subseteq \D_E^0$, and let $l$ be a positive 
integer such that $\Xi_l=\Xi_u$. Let $b\in\F_E^{k-1}$ and let $\alpha,\beta\in E^l$. Then 
$T_{\alpha,\beta}(b)$ belongs to $\D_E^0$ and thus it commutes with $\varphi^m(u)$ for all $m$, 
since $u$ commutes with the vertex projections. Consequently, for each $r\geq 1$ we have 
$$ \begin{aligned}
\Ad(u^*_{l+r})(b) &= \Ad(\varphi^{l-1+r}(u^*)\cdots \varphi^l(u^*)) \Big( \sum_{\alpha,\beta\in E^l}
S_\alpha T_{\alpha,\beta}(b) S_\beta^* \Big) \\ 
 & = \sum_{\alpha,\beta\in E^l} S_\alpha \Ad(\varphi^{r-1}(u^*)
\cdots u^*)(T_{\alpha,\beta}(b)) S_\beta^* \\
 & = \sum_{\alpha,\beta\in E^l} S_\alpha T_{\alpha,\beta}(b) S_\beta^* \ . 
\end{aligned} $$
Thus for each $b\in\F_E^{k-1}$ the sequence $\Ad(u^*_m)(b)$ stabilizes from $m=l+1$. Write 
$u^*=\sum_{e,f\in E^1} S_e b_{e,f} S_f^*$, for some $b_{e,f}\in\F_E^{k-1}$. Then 
for each $m$ we have 
$$ \begin{aligned} 
\Ad(u^*_{m+1})(u^*) & = \sum_{e,f\in E^1}\Ad(\varphi(\varphi^{m-1}(u^*)\cdots 
   \varphi(u^*)u^*))(S_e b_{e,f} S_f^*)  \\ 
 & = \sum_{e,f\in E^1} S_e \Ad(\varphi^{m-1}(u^*)\cdots \varphi(u^*)u^*)(b_{e,f}) S_f^* 
\end{aligned} $$
and, consequently, the sequence $\Ad(u^*_m)(u^*)$ stabilizes from $m=l+2$. 

(2) $\Rightarrow$ (1):  Suppose that the sequence ${\rm Ad}(u_m^*)(u^*)$ eventually stabilizes. 
Hence ${\rm Ad}(u_r^*)(u^*)=w$ for all sufficiently large $r$. It follows that 
$\lambda_u(w)={\rm Ad}(u_r)(w)=u^*$ for suitable $r$ depending on $w$. 
Thus $\lambda_u(w)u=u^*u=1$ and, consequently, 
$\lambda_u$ is invertible with inverse $\lambda_w$.  This completes the proof. 

We also include a different and much more direct proof of implication (1) $\Rightarrow$ (2), that 
sheds additional light on the equivalent conditions of the theorem and is interesting in its own right. 

Let $\lambda_u$ be invertible, and suppose that there exists an $l\in\bN$ and 
a unitary $v\in\U_E$ in $\F_E^l$  such that $\lambda_u \lambda_v = {\rm id}$. 
Then we have $\lambda_u(v)u = 1$.
Since $v\in \F_E^l$, $u^*=\lambda_u(v)={\rm Ad}(u_l)(v)$, and hence ${\rm Ad}(u_l^*)(u^*)=v$.
Now for $r\geq 1$ we have 
\begin{eqnarray*}
{\rm Ad}(u_{l+r}^*)(u^*) & = & {\rm Ad}(\varphi^{l+r-1}(u^*)\cdots \varphi(u^*)u^*)(u^*)\\
&=& \varphi^{l+r-1}(u^*)\cdots \varphi^l(u^*){\rm Ad}(u_l^*)(u^*)\varphi^l(u)\cdots\varphi^{l+r-1}(u)\\
&=& \varphi^{l+r-1}(u^*)\cdots \varphi^l(u^*)v\varphi^l(u)\cdots\varphi^{l+r-1}(u) \\
&=& v
\end{eqnarray*}
since $v$ commutes with $\varphi^m(u)$ for every $m\geq l$. Thus we can conclude that
${\rm Ad}(u_m^*)(u^*)$ stabilizes at $v$ from $m=l$.
\end{proof}

\begin{remark}\label{bogolubov}
\rm Let $u\in\F_E^1\cap\U_E$, so that $\lambda_u$ is quasi-free. Since $\F_E^0=\D_E^0$, we have  
$V_1=\D_E^0/\D_E^0=\{0\}$ and consequently each $\tilde{a}^u_{e,f}$ is a zero map. Therefore 
$A_u=\{0\}$ and Theorem \ref{main} trivially implies that $\lambda_u$ is an automorphism of $C^*(E)$. 
\end{remark}

If $u\in\U_E$ normalizes $\D_E$ then $\lambda_u(\D_E)\subseteq\D_E$. It may well happen that 
such a restriction is an automorphism of $\D_E$ even though $\lambda_u$ is not invertible. For 
unitaries in the algebraic part of $\F_E$ this can be checked in a way similar to \cite[Theorem 3.4]{CS}. 
Indeed, let $u\in\U_E\cap\N_{\F_E^k}(\D_E^k)$. Then it follows from (\ref{commutation}) that 
$u\in\N_{\F_E}(\D_E)$. Furthermore, the subspace $\D_E^{k-1}$ of $\F_E^{k-1}$ 
is invariant under the action of all maps $a^u_{e,f}$, $e,f\in E^1$. We denote by $b^u_{e,f}$ 
the restriction of $a_{e,f}^u$ to $\D_E^{k-1}$, and by $\tilde{b}_{e,f}^u$ the map induced on 
$V_k^D:=\D_E^{k-1}/\D_E^0$. Let $A_u^D$ be the subring of ${\mathcal L}(V_k^D)$ generated 
by $\{\tilde{b}_{e,f}^u : e,f\in E^1\}$. Also, we consider a nested sequence of subspaces  
$\Xi_{r}^D$ of $\D_E^{k-1}$, defined inductively as 
\begin{equation}\label{kxid}
\Xi_0^D = \D_E^{k-1}, \;\;\; \Xi_r^D = \lambda_u(H)^* \Xi_{r-1}^D \lambda_u(H), \; r\geq1. 
\end{equation}
Each $\Xi_r^D$ is finite dimensional 
and self-adjoint. We set $\Xi_u^D:=\bigcap_r \Xi_r^D$. 

\begin{theorem}\label{maindiagonal}
Let $E$ be a finite graph without sinks and let $u\in\U_E\cap\N_{\F_E^k}(\D_E^k)$, for 
some $k\geq1$. Then the following conditions are equivalent:
\begin{itemize}
\item[(1)] $\lambda_u$ restricts to an automorphism of $\D_E$;
\item[(2)] the ring $A_u^D$ is nilpotent;
\item[(3)] $\Xi_u^D \subseteq \D_E^0$.
\end{itemize} 
\end{theorem}
\begin{proof}
(1) $\Rightarrow$ (3): Since the algebraic part $\bigcup_{t=0}^\infty \D_E^t$ of $\D_E$ coincides 
with the linear span of all projections in $\D_E$, every automorphism of $\D_E$ restricts to an 
automorphism of $\bigcup_{t=0}^\infty \D_E^t$. Thus, there exists an $l$ such that 
$(\lambda_u|_{\D_E})^{-1}(\D_E^{k-1})\subseteq \D_E^l$. Let $\alpha=(e_1,e_2,\cdots,e_l)$ 
and $\beta=(f_1,f_2,\cdots, f_l)$ be in $E^l$, and let $R_{\alpha, \beta}:=b_{e_l,f_l}^u\cdots 
b_{e_2,f_2}^u b_{e_1,f_1}^u$ be in $(A_u^D)^l$. Then the same argument as in the proof of 
implication (1)$\Rightarrow$(3) in Theorem \ref{main} yields that $R_{\alpha, \beta}(d)\in\D_E^0$ 
for all $d\in\D_E^{k-1}$. Thus $(A_u^D)^l=\{0\}$. But as in the proof of implication 
(3)$\Rightarrow$(4) in Theorem \ref{main}, this implies that $\Xi_u^D\subseteq\D_E^0$. 

(3) $\Rightarrow$ (2): Let $\Xi_u^D \subseteq \D_E^0$ and let $l$ be a positive 
integer such that $\Xi_l^D=\Xi_u^D$. Let $d\in\D_E^{k-1}$ and let $\alpha,\beta\in E^l$. Then 
$R_{\alpha,\beta}(d)$ belongs to $\D_E^0$, and this entails $(A_u^D)^l=\{0\}$, i.e. 
the ring $A_u^D$ is nilpotent. 

(2) $\Rightarrow$ (1): Suppose that $A_u^D$ is nilpotent. We show by induction on $r\geq k$ 
that all $\D_E^r$ are in the range of $\lambda_u$ restricted to $\bigcup_{t=0}^\infty \D_E^t$. 

Firstly, let $r=k$ and $d\in\D_E^k$. Similarly to the argument in the implication 
(4)$\Rightarrow$(2) of the proof of Theorem \ref{main} one shows that the sequence 
$\Ad(\varphi^m(u^*)\cdots\varphi(u^*)u^*)(d)$ eventually stabilizes at some 
$f\in\bigcup_{t=0}^\infty \D_E^t$. It then follows that $d=\lambda_u(f)$. 

For the inductive step, suppose that $r\geq k$ and $\D_E^r\subseteq \lambda_u(\bigcup_{t=0}^\infty 
\D_E^t)$. Since $\D_E^{r+1}$ is generated by $\D_E^r$ and $\varphi^r(\D_E^1)$, it suffices to 
show that $\varphi^r(y)$ belongs to $\lambda_u(\bigcup_{t=0}^\infty \D_E^t)$ for all $y\in\D_E^1$. 
However, $\varphi^r(y)$ commutes with $u$ and $\varphi^{r-1}(y)\in\D_E^r$ is 
in $\lambda_u(\bigcup_{t=0}^\infty \D_E^t)$ by the inductive hypothesis. Thus the sequence 
$$ \Ad(\varphi^m(u^*)\cdots\varphi(u^*)u^*)(\varphi^r(y)) = 
    \varphi(\Ad(\varphi^{m-1}(u^*)\cdots\varphi(u^*)u^*)(\varphi^{r-1}(y))) $$ 
eventually stabilizes at $\lambda_u^{-1}(\varphi^r(y))\in\bigcup_{t=0}^\infty \D_E^t$. 
\end{proof}

It should be noted that, in the setting of Theorem \ref{maindiagonal}, it may well happen that 
$\lambda_u|_{\D_E}$ is an automorphism of $\D_E$ while $\lambda_u$ is a proper endomorphism 
of $C^*(E)$. In that case there may not exist any unitary $w\in C^*(E)$ such that 
$(\lambda_u|_{\D_E})^{-1}=\lambda_w|_{\D_E}$, and thus $(\lambda_u|_{\D_E})^{-1}$ is not 
localized in our sense (as defined in the first paragraph of this section). See 
\cite{CS}, \cite{CKS}, \cite[Theorem 3.8]{CS2} and \cite[Proposition 3.2]{CHS3} for 
examples and further discussion of this interesting point. 


\section{The permutative automorphisms}\label{sectionP}

Throughout this section, we assume that $E$ is a {\em finite graph without sinks}. 
Our main goal in this section is to give a combinatorial 
criterion for invertibility of $\lambda_u$, $u\in\P_E$, analogous to \cite[Corollary 4.12]{CHS1}. 
Both the statement of the criterion (see Theorem \ref{combinatorialcriterion} below) and its proof are 
quite similar to those given in the case of the Cuntz algebras in \cite{CS}, suitably generalized to the 
present case of graph $C^*$-algebras. The key idea is to break the process into two steps,  
{\em Condition} (b) and {\em Condition} (d), of which the former detects those 
endomorphisms which restrict to automorphisms of the diagonal $\D_E$. 

It is useful to look at collections of paths of a fixed length beginning or ending 
at the same vertex. Hence we introduce the following notation. 
For $v,w\in E^0$ and $k\in\bN$, let $E_{v,*}^k:=\{\alpha\in E^k:r(\alpha)=v\}$, 
$E^k_{*,v}:=\{\alpha\in E^k : s(\alpha)=v\}$  and $E_{v,w}^k:=
\{\alpha\in E^k:r(\alpha)=v,\; s(\alpha)=w\}$. Then $E^k=\bigcup_{v\in E^0} E_{*,v}^k=
\bigcup_{v\in E^0}E_{v,*}^k=\bigcup_{v,w\in E^0}E_{v,w}^k$, disjoint unions.
If $u\in\P_E^k$, $k>0$, then there exist permutations $\sigma_v\in\perm(E_{v,*}^k)$ such that 
\begin{equation}\label{upk}
u = \sum_{v\in E^0} \sum_{\alpha\in E_{v,*}^k} S_{\sigma_v(\alpha)}S^*_\alpha. 
\end{equation}
A unitary $u\in\P_E^k$, $k>0$, commutes with all the vertex projections if and only if there 
exist permutations $\sigma_{v,w}\in\perm(E_{v,w}^k)$ such that 
\begin{equation}\label{upke}
u = \sum_{v,w\in E^0} \sum_{\alpha\in E_{v,w}^k} S_{\sigma_{v,w}(\alpha)}S^*_\alpha. 
\end{equation}
If the unitary $u\in\P_E^k\cap\U_E$ is understood, as in equation (\ref{upke}), 
then we will denote by $\sigma=\cup_{v,w\in E^0}\sigma_{v,w}$ the corresponding 
permutation of $E^k$. In that case, we will also write $\lambda_u=\lambda_\sigma$. 

Now let $u\in\P_E^k\cap\U_E$, $e,f\in E^1$, and consider the linear map $a^u_{e,f}$, 
as defined in (\ref{aefu}). With 
respect to the basis $\{S_\mu S_\nu^*:\mu,\nu\in E^{k-1}\}$ of $\F_E^{k-1}$ so ordered 
that the initial vectors span $\D_E^{k-1}$, the matrix of $a^u_{e,f}$ has the block form 
\begin{equation}\label{amatrix}
a^u_{e,f} = \left( \begin{array}{cc} b^u_{e,f} & c^u_{e,f} \\ 0 & d^u_{e,f} \end{array} \right), 
\end{equation}
similarly to \cite[Section 4]{CS}. The first block corresponds to the subspace $\D_E^{k-1}$ of 
$\F_E^{k-1}$. Thus, the map $\tilde{a}^u_{e,f}\in\LL(V_k)$ has a matrix 
\begin{equation}\label{atildematrix}
\tilde{a}^u_{e,f} = \left( \begin{array}{cc} \tilde{b}^u_{e,f} & * \\ 
0 & d^u_{e,f} \end{array} \right), 
\end{equation}
with the first block corresponding to the subspace $V_k^D$ of $V_k$. Note that the passage 
from the space $\F_E^{k-1}$ to its quotient $V_k$ does not affect the matrix for 
$d_{e,f}^u$ and thus there is no tilde over it in formula (\ref{atildematrix}). 
It is an immediate corollary to Theorem \ref{main} that an endomorphism $\lambda_u$ of $C^*(E)$ 
is invertible if and only if the following two conditions are satisfied.  
\begin{center}
\obeylines
Condition (b): \hspace{2mm} the ring generated by $\{\tilde{b}^u_{e,f} : e,f\in E^1\}$ is nilpotent.  
Condition (d): \hspace{2mm} the ring generated by $\{d^u_{e,f} : e,f\in E^1\}$ is nilpotent. 
\end{center}
The remainder of this section is devoted to the description of a convenient combinatorial interpretation 
of these two crucial conditions,  similar  to the one appearing in \cite{CS} and used in the analysis of 
permutative endomorphisms of the Cuntz algebras. 
By virtue of Theorem \ref{maindiagonal}, Condition (b) alone is equivalent to the restriction of 
$\lambda_u$ to the diagonal $\D_E$ being an automorphism. 


\subsection{Condition (b)}

We fix $u\in\P_E^k\cap\U_E$ and denote by $\sigma$ the corresponding permutation, as above. If 
$e\neq f$ then $b^u_{e,f}=0$. Thus, it suffices to consider the ring generated by maps 
$b^u_e:=b^u_{e,e}$, $e\in E^1$. Since $b^u_e(1)=P_{r(e)}$, the matrix of $b^u_e$ has 
exactly one 1 in the row corresponding to each $\alpha\in E^{k-1}_{*,r(e)}$, and 0's elsewhere. 
Consequently, each $b^u_e$ may be identified with a mapping 
\begin{equation}\label{fe}
f_e^u: E^{k-1}_{*,r(e)} \to E^{k-1}_{*,s(e)}, \;\;\; f_e^u(\alpha)=\beta, 
\end{equation}
whenever $b^u_e$ has 1 in the $\alpha$--$\beta$ entry. If the unitary $u$ is given by a 
permutation $\sigma$ then 
\begin{equation}\label{fepermutation}
f_e^u(\alpha)=\beta \; \Leftrightarrow \; \exists g\in E^1 \; \text{s.t.} \; \sigma(e,\alpha)=(\beta,g). 
\end{equation}
The product $b^u_e b^u_g$ corresponds to the composition $f^u_g\circ f^u_e$ (in 
reversed order of $e$ and $g$). Now Condition (b) may be phrased in terms of mappings 
$\{f^u_e\}$ rather than $\{b^u_e\}$, as follows: 

\vspace{2mm}\noindent
There exists an $m$ such that for all $e_1,\ldots,e_m\in E^1$ if $T=f^u_{e_1}\circ\ldots\circ f^u_{e_m}$ 
then for all $v\in E^0$ and $\alpha\in E^{k-1}$ either $E^{k-1}_{*,v}\cap T^{-1}(\alpha)=
\emptyset$ or $E^{k-1}_{*,v}\subseteq T^{-1}(\alpha)$. 

\vspace{2mm}\noindent
Taking into account (\ref{fe}) above, we arrive at the following: 

\vspace{2mm}\noindent
{\bf Condition (b)}: There exists an integer $m\in\bZ$ such that for all $e_1,\ldots,e_m\in E^1$ either: (i)  
$f^u_{e_1}\circ\ldots\circ f^u_{e_m}$ has the empty domain, or (ii) its domain equals 
$E^{k-1}_{*,r(e_m)}$ and its range consists of exactly one element. 

\vspace{2mm}
In the remainder of this section, notation $(\alpha,\beta)$ indicates either a single path in $E^*$ or 
an ordered pair in the cartesian product $E^*\times E^*$. This will be clear from context. 

\begin{lemma}\label{conditionb}
Let $E$ be a finite graph without sinks in which every loop has an exit. Let $u\in\P_E^k\cap\U_E$. 
Then Condition (b) holds for $u$ (and hence $\lambda_u|_{\D_E}$ is an automorphism of $\D_E$) 
if and only if there exists a partial order $\leq$ on 
$\bigcup_{v\in E^0}E^{k-1}_{*,v}\times E^{k-1}_{*,v}$ such that:
\begin{enumerate}
\item if $v\in E^0\setminus r(E^1)$ then each element of $E^{k-1}_{*,v}\times E^{k-1}_{*,v}$ 
is minimal, each diagonal element $(\alpha,\alpha)$ is minimal, and there are no other minimal elements; 
\item if $e\in E^1$ and $\alpha\neq\beta\in E^{k-1}_{*,r(e)}$ then 
$(f^u_e(\alpha),f^u_e(\beta))\leq(\alpha,\beta)$. 
\end{enumerate}
\end{lemma}
\begin{proof}
At first suppose that Condition (b) holds for $u$. 
Define a relation $\leq$ as follows. For any $\alpha\in E^{k-1}$ set $(\alpha,\alpha)\leq(\alpha,
\alpha)$. If $\gamma \neq \delta$ then
$(\alpha,\beta) \leq (\gamma,\delta) $ if and only if there exists a sequence
$e_1,\ldots,e_d\in E^1$, possibly empty, such that $\alpha = (f^u_{e_1} \circ \cdots \circ
f^u_{e_d})(\gamma)$ and $\beta =( f^u_{e_1} \circ \cdots \circ f^u_{e_d})(\delta)$.
Reflexivity and transitivity of $\leq$ are obvious. To see that $\leq$ is also antisymmetric, 
suppose that $(\alpha,\beta) \leq (\gamma,\delta)$ and $(\gamma,\delta) \leq (\alpha,\beta)$.
If $(\alpha,\beta)\neq(\gamma,\delta)$ then, by definition of $\leq$,
$\alpha \neq \beta$, $\gamma \neq \delta$ and there exist edges 
$e_1,\ldots,e_d$, $g_1,\ldots,g_h$ such that
$(\alpha,\beta) = (f^u_{e_1} \circ \cdots \circ f^u_{e_d})(\gamma,\delta)$ and
$(\gamma,\delta) = (f_{g_1} \circ \cdots \circ f_{g_h})(\alpha,\beta)$.
Then $(\alpha,\beta) = (f^u_{e_1} \circ \cdots \circ f^u_{e_d} \circ f_{g_1} \circ \cdots
\circ f_{g_h})(\alpha,\beta)$. That is, $f^u_{e_1} \circ \cdots
\circ f^u_{e_d} \circ f_{g_1} \circ \cdots \circ f_{g_h}$ has
two distinct fixed points,  a contradiction with Condition (b). Thus $(\alpha,\beta) = 
(\gamma,\delta)$ and $\leq$ is also antisymmetric. Hence $\leq$ is a partial order 
satisfying condition 2 above. By the very definition of $\leq$, 
if $v\in E^0\setminus r(E^1)$ then each element of $E^{k-1}_{*,v}\times E^{k-1}_{*,v}$ 
is minimal, and likewise each diagonal element $(\alpha,\alpha)$ is minimal. If any other element 
were minimal for $\leq$ then there would exist $\alpha\neq\beta$ and $e\in E^1$ such that 
$f^u_e(\alpha)=\alpha$ and $f^u_e(\beta)=\beta$. Thus $f^u_e$ would have two distinct 
fixed points, contradicting Condition (b). Thus condition 1 holds true as well. 

Conversely, if a partial order $\leq$ with the required properties exists, then counting shows that each 
sufficiently long composition product of mappings $\{f^u_e\}$ either has the empty domain or its 
range consists of a single element (and the domain is as required, due to (\ref{fe})). 
This completes the proof. 
\end{proof}

\begin{remark}\label{trees}\rm 
By the diagram of $f^u_e$ we mean a directed graph with vertices corresponding to the union of 
the domain and the range of the map $f^u_e$ and with an edge from vertex $\alpha$ to $\beta$ if and only 
if $f^u_e(\alpha)=\beta$. Combining the diagrams of all $f^u_e$, $e\in E^1$, we obtain a directed graph 
whose edges are labelled by $\{f^u_e\}$ or simply by the edges of $E$, see Example \ref{example1} below. 
Our Condition (b) is equivalent to 
existence of a positive integer $m$ such that all words of length $m$ are {\em synchronizing} for 
this labelled graph\footnote{We are grateful to Rune Johansen for pointing this out.}. 
Also note that if the conditions of Lemma \ref{conditionb} are satisfied and $e\in E^1$ is such that 
$s(e)=r(e)$, then the diagram of $f^u_e$ is a rooted tree with the root 
being the unique fixed point, cf. \cite[Section 4.1]{CS}. 
\end{remark}


\subsection{Condition (d)}

Again, we fix a $u\in\P_E^k\cap\U_E$ and denote by $\sigma$ the corresponding permutation. 
It is easy to verify that for $e,g\in E^1$ each row of the matrix $d^u_{e,g}$  
either may have 1 in one place and 0's elsewhere 
or may consist of all zeros. This matrix has 1 in $(\alpha,\beta)$ row and $(\gamma,\delta)$ column 
if and only if there exists an $h\in E^1$ such that $S_\alpha S^*_\beta=S^*_e u^* S_\gamma 
S_h S^*_h S^*_\delta u S_g$. In turn, this takes place if and only if 
\begin{equation}\label{dsigma}
\sigma(e,\alpha)=(\gamma,h) \;\;\; \text{and} \;\;\; \sigma(g,\beta)=(\delta,h). 
\end{equation}
For each $e,g\in E^1$ we now define a mapping $f^u_{e,g}$, as follows. The domain $D(f^u_{e,g})$ 
of $f^u_{e,g}$ consists of all $(\alpha,\beta)\in E^{k-1}_{*,r(e)}\times E^{k-1}_{*,r(g)}$ for 
which the $(\alpha,\beta)$ row of $d^u_{e,g}$ is non-zero, and the corresponding value is 
$f^u_{e,g}(\alpha,\beta)=(\gamma,\delta)\in E^{k-1}\times E^{k-1}$ for $(\gamma,\delta)$ 
satisfying (\ref{dsigma}). Note that, by the very definition of $d^u_{e,g}$, in such a case we must 
necessarily have $\alpha\neq\beta$ and $\gamma\neq\delta$. We denote $\Psi_u:=E^{k-1}\times 
E^{k-1} \setminus \{(\alpha,\alpha) : \alpha\in E^{k-1}\}$. We also denote by $\Delta_u$ the 
subset of $\Psi_u$ consisting of all those $(\alpha,\beta)$ for which there exist $e,g\in E^1$ such that 
$(\alpha,\beta)$ belongs to the domain $D(f^u_{e,g})$. 

It is a simple matter to verify that in terms of mappings $\{f^u_{e,g}\}$ Condition (d) may 
be rephrased as follows (cf. \cite[Section 4.3]{CS}). 

\vspace{2mm}\noindent
{\bf Condition (d)}: There exists an $m$ such that for all 
$(e_1,g_1),\ldots,(e_m,g_m)\in \Psi_u$ the 
domain of the map $f^u_{e_1,g_1}\circ\ldots\circ f^u_{e_m,g_m}$ is empty.

\vspace{2mm}
The proof of the following lemma is essentially the same as that of \cite[Lemma 4.10]{CS} and thus 
it is omitted. 

\begin{lemma}\label{conditiond}
Let $E$ be a finite graph without sinks in which every loop has an exit. Let $u\in\P_E^k\cap\U_E$. 
Then Condition (d) holds for $u$ if and only if there exists a partial order $\leq$ on $\Psi_u$ such that: 
\begin{enumerate}
\item the set of minimal elements coincides with $\Psi_u\setminus\Delta_u$; 
\item if $e,g\in E^1$ and $(\alpha,\beta)\in D(f^u_{e,g})$ then $f^u_{e,g}(\alpha, \beta) 
\leq (\alpha,\beta)$. 
\end{enumerate}
\end{lemma}

Combining Lemma \ref{conditionb} with Lemma \ref{conditiond} we obtain a combinatorial 
criterion of invertibility of permutative endomorphisms, similar to \cite[Corollary 4.12]{CS}. 

\begin{theorem}\label{combinatorialcriterion}
Let $E$ be a finite graph without sinks in which every loop has an exit, and let $u\in\P_E^k\cap\U_E$. 
Then the endomorphism $\lambda_u$ is invertible if and only if conditions of Lemma 
\ref{conditionb} and Lemma \ref{conditiond} hold for $u$. 
\end{theorem}

\subsection{Examples}

We give two examples with small graphs illustrating the combinatorial machinery developed 
in the preceding section. In Example \ref{example1}, we exhibit a proper permutative endomorphism 
of $C^*(E)$ which restricts to an automorphism of the diagonal MASA $\D_E$. On the other 
hand, in Example \ref{example2} we provide an order 2 permutative automorphism of a Kirchberg 
algebra $C^*(E)$ with $K_0(C^*(E))\cong \bZ \cong K_1(C^*(E))$, which is neither quasi-free 
nor comes from a graph automorphism. 

\begin{example}\label{example1}
{\rm 

Consider the following graph $E$ (the Fibonacci graph): 

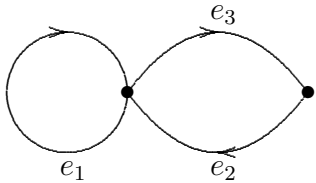
\[ \beginpicture
\setcoordinatesystem units <0.8cm,0.8cm>
\setplotarea x from -5 to 1, y from -2 to 2
\put {$\bullet$} at -3 0
\put {$\bullet$} at 0 0
\setquadratic
\plot -3 0  -1.5 1  0 0 /
\plot -3 0  -1.5 -1  0 0 /
\circulararc 360 degrees from -3 0 center at -4 0
\arrow <0.235cm> [0.2,0.6] from -4.1 1 to -4 1
\arrow <0.235cm> [0.2,0.6] from -1.6 1 to -1.5 1
\arrow <0.235cm> [0.2,0.6] from -1.4 -1 to -1.5 -1
\put{$e_1$} at -3.9 -1.3
\put{$e_2$} at -1.4 -1.3
\put{$e_3$} at -1.4 1.3
\endpicture \] 

At level $k=2$, there are 2 permutations in $\P_E^2\cap\U_E\cong\bZ_2$. Denoting edge $e_j$ simply 
by $j$, the non-trivial transposition is $(11,32)$. The corresponding map $f_1:\{1,3\} \to \{1,3\}$ 
is such that $f_1(1)=3$ and $f_1(3)=1$. Thus Condition (b) does not hold and, consequently, 
the corresponding endomorphism is surjective neither on $C^*(E)$ nor on $\D_E$. 

At level $k=3$, there are 24 permutations in $\P_E^3\cap\U_E\cong S_3\times\bZ_2\times\bZ_2$. 
These are generated by $\sigma=(132,321)$, $\tau=(111,132,321)$, $\upsilon=(113,323)$ and 
$\omega=(211,232)$. Of these 24 permutations, 
only $\tau\upsilon$ and $\id$ satisfy Condition (b). However,  $\tau\upsilon$ 
does not satisfy Condition (d). Thus, $\lambda_{\tau\upsilon}$ is a proper endomorphism of $C^*(E)$ 
(i.e., it is not surjective) which restricts to an automorphism of $\D_E$. 

The maps $f_1:\{(11),(13),(32)\} \to \{(11),(13),(32)\}$, 
$f_2:\{(11),(13),(32)\} \to \{(21),(23)\}$, and 
$f_3:\{(21),(23)\} \to \{(11),(13),(32)\}$ 
corresponding to permutation $\tau\upsilon$ and involved in verification 
of Condition (b) are illustrated by the following labelled graph. 

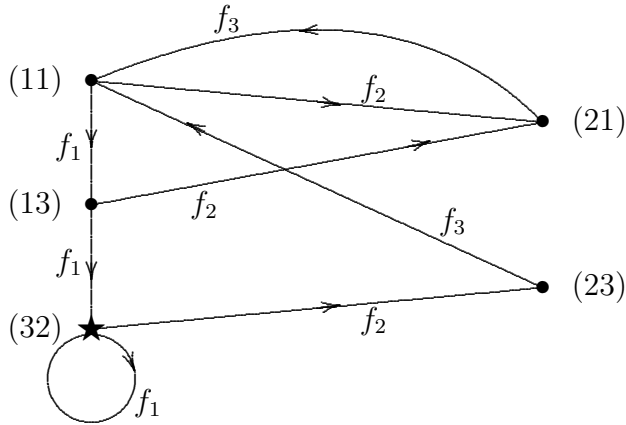
\[ \beginpicture
\setcoordinatesystem units <1.5cm,1.1cm>
\setplotarea x from -2 to 5, y from -3.5 to 2.5
\put {$\bigstar$} at -1 -1.5
\put {$\bullet$} at -1 0
\put {$\bullet$} at -1 1.5
\put {$\bullet$} at 3 1 
\put {$\bullet$} at 3 -1
\circulararc 360 degrees from -1 -1.57 center at -1 -2.1
\arrow <0.235cm> [0.2,0.6] from -0.655 -1.88 to -0.62 -2
\arrow <0.235cm> [0.2,0.6] from -1 0.9 to -1 0.7
\arrow <0.235cm> [0.2,0.6] from -1 -0.7 to -1 -0.9
\arrow <0.235cm> [0.2,0.6] from 1 -1.25 to 1.2 -1.21
\arrow <0.235cm> [0.2,0.6] from 1.8 0.7 to 2 0.77
\arrow <0.235cm> [0.2,0.6] from 1 1.25 to 1.2 1.21
\arrow <0.235cm> [0.2,0.6] from 0 0.89 to -0.15 0.99
\arrow <0.235cm> [0.2,0.6] from 1 2.12 to 0.85 2.1
\put{$\small (11)$} at -1.5 1.5
\put{$\small (13)$} at -1.5 0
\put{$\small (32)$} at -1.5 -1.5
\put{$\small (21)$} at 3.5 1
\put{$\small (23)$} at 3.5 -1
\put{$f_1$} at -1.2  0.7
\put{$f_1$} at -1.2  -0.7
\put{$f_1$} at -0.5 -2.4
\put{$f_2$} at  1.5 1.4
\put{$f_2$} at  0 0
\put{$f_2$} at  1.5 -1.4
\put{$f_3$} at  0.2  2.25
\put{$f_3$} at  2.2  -0.2

\setlinear 
\plot -1 1.5  -1 -1.5   3 -1  -1 1.5   3 1  -1 0  /
\setquadratic
\plot -1 1.5  1.3 2.1   3 1 /
\endpicture \] 

Note that the diagram of the map $f_1$, corresponding to an edge whose source and range 
coincide, gives rise to a rooted tree. This is the left hand side of the diagram above, with the root 
(the unique fixed point for $f_1$) indicated by a star. 

}
\end{example}

\begin{example}\label{example2}
{\rm 

Consider the following graph $E$: 

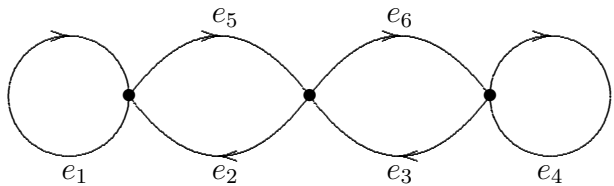
\[ \beginpicture
\setcoordinatesystem units <0.8cm,0.8cm>
\setplotarea x from -5 to 5, y from -2 to 2
\put {$\bullet$} at -3 0
\put {$\bullet$} at 0 0
\put {$\bullet$} at 3 0
\setquadratic
\plot -3 0  -1.5 1  0 0 /
\plot 0 0  1.5 1  3 0 /
\plot -3 0  -1.5 -1  0 0 /
\plot 0 0  1.5 -1  3 0 /
\circulararc 360 degrees from -3 0 center at -4 0
\circulararc 360 degrees from 3 0 center at 4 0
\arrow <0.235cm> [0.2,0.6] from -4.1 1 to -4 1
\arrow <0.235cm> [0.2,0.6] from -1.6 1 to -1.5 1
\arrow <0.235cm> [0.2,0.6] from 1.4 1 to 1.5 1
\arrow <0.235cm> [0.2,0.6] from 3.9 1 to 4 1
\arrow <0.235cm> [0.2,0.6] from -1.4 -1 to -1.5 -1
\arrow <0.235cm> [0.2,0.6] from 1.6 -1 to 1.5 -1
\put{$e_1$} at -3.9 -1.3
\put{$e_2$} at -1.4 -1.3
\put{$e_3$} at 1.5 -1.3
\put{$e_4$} at 4 -1.3
\put{$e_5$} at -1.4 1.3
\put{$e_6$} at 1.5 1.3
\endpicture \] 

At level $k=2$, there are 8 permutations in $\P_E^2\cap\U_E\cong\bZ_2^3$. Denoting edge $e_j$ simply 
by $j$, these are generated by transpositions $\sigma=(25,63)$, $\tau= (11,52)$ and $\upsilon=(36,44)$. 
Of these 8 permutations, only $\sigma$ and $\id$ satisfy Condition (b). Since $\sigma$ satisfies Condition (d) 
as well, $\lambda_\sigma$ is an automorphism of $C^*(E)$. We have 
\begin{eqnarray*}
\lambda_\sigma(S_2) & = & S_6 S_3 S_5^* + S_2 S_1 S_1^*, \\
\lambda_\sigma(S_6) & = & S_2 S_5 S_3^* + S_6 S_4 S_4^*, \\
\lambda_\sigma(S_j) & = & S_j, \;\;\; j=1,3,4,5,  
\end{eqnarray*}  
and it follows immediately that $\lambda_\sigma^2=\id$. 

We note that in the present case $K_0(C^*(E))\cong \bZ \cong K_1(C^*(E))$ by \cite{Cun3} and 
\cite[Section 4]{KPRR}, see also \cite{Rae}. Thus $C^*(E)$ is not isomorphic to a Cuntz algebra, and 
hence this example is not covered in any way by the results of \cite{CS}. 

}
\end{example}


\smallskip\noindent
Roberto Conti \\
Dipartimento di Scienze di Base e Applicate per l'Ingegneria \\
Sezione di Matematica \\
Sapienza Universit\`a di Roma \\
Via A. Scarpa 16 \\
00161 Roma, Italy \\
E-mail: roberto.conti@sbai.uniroma1.it \\

\smallskip\noindent
Jeong Hee Hong \\
Department of Data Information \\
Korea Maritime University \\
Busan 606--791, South Korea \\
E-mail: hongjh@hhu.ac.kr \\

\smallskip \noindent
Wojciech Szyma{\'n}ski\\
Department of Mathematics and Computer Science \\
The University of Southern Denmark \\
Campusvej 55, DK-5230 Odense M, Denmark \\
E-mail: szymanski@imada.sdu.dk

\end{document}